\documentclass[10pt]{article}
\usepackage{amssymb}
\usepackage{amsmath}
\usepackage{fullpage}
\usepackage[T1]{fontenc}

\newcommand{\abs}[1]{\left| #1 \right|}

\newcommand{\w}[0]{\omega}
\newcommand{\N}[0]{\mathbb{N}}
\newcommand{\Z}[0]{\mathbb{Z}}

\newcommand{\scr}[1]{ \mathcal{ #1 } }
\long\def\symbolfootnote[#1]#2{\begingroup%
\def\thefootnote{\fnsymbol{footnote}}\footnote[#1]{#2}\endgroup} 
\newtheorem{theorem}{Theorem}[section]
\newtheorem{lemma}[theorem]{Lemma}
\newtheorem{proposition}[theorem]{Proposition}
\newtheorem{corollary}[theorem]{Corollary}
\newtheorem{claim}[theorem]{Claim}

\newenvironment{proof}[1][Proof]{\begin{trivlist}
\item[\hskip \labelsep {\bfseries #1}]}{\end{trivlist}}
\newenvironment{definition}[1][Definition]{\begin{trivlist}
\item[\hskip \labelsep {\bfseries #1}]}{\end{trivlist}}
\newenvironment{example}[1][Example]{\begin{trivlist}
\item[\hskip \labelsep {\bfseries #1}]}{\end{trivlist}}

\newcommand{\qed}{\hfill \mbox{\raggedright \rule{.07in}{.1in}}}

\title{Permutation Complexity of the Thue-Morse Word}
\author{Steven Widmer\footnote{The Mathematics Institute, Reykjavik University, Menntavegi 1, IS-101 Reykjavik, ICELAND (s.widmer1@gmail.com).}}
\date{}
\begin{document}

\maketitle


\begin{abstract}
Given a countable set $X$ (usually taken to be $\N$ or $\Z$), an infinite permutation $\pi$ of $X$ is a linear ordering $\prec_\pi$ of $X$, introduced in $\cite{FlaFrid}$.  This paper investigates the combinatorial complexity of the infinite permutation on $\N$ associated with the well-known and well-studied Thue-Morse word.  A formula for the complexity is established by studying patterns in subpermutations and the action of the Thue-Morse morphism on the subpermutations.
$\vspace{2.0ex}$

$\textbf{Keywords:}$ infinite permutation, permutation complexity, Thue-Morse word
\end{abstract}

\section{Introduction}

%
Permutation complexity of aperiodic words is a relatively new notion of word complexity which was first introduced and studied by Makarov $\cite{Makar06}$ based on ideas of S.V. Avgustinovich (see the acknowledgements in $\cite{FlaFrid}$), and is based on the idea of an infinite permutation associated to an aperiodic word.  For an infinite aperiodic word $\w$, no two shifts of $\w$ are identical.  Thus, given a linear order on the symbols used to compose $\w$, no two shifts of $\w$ are equal lexicographically.  The infinite permutation associated with $\w$ is the linear order on $\N$ induced by the lexicographic order of the shifts of $\w$.  The permutation complexity of the word $\w$ will be the number of distinct subpermutations of a given length of the infinite permutation associated with $\w$.

Infinite permutations associated with infinite aperiodic words over a binary alphabet act fairly well-behaved, but many of the arguments used for binary words break down when used with words over more than two symbols.  Given a subpermutation of length $n$ of an infinite permutation associated with a binary word, a portion of length $n-1$ of the word can be recovered from the subpermutation.  This is not always the case for subpermutations associated with words over 3 or more symbols.  For example, consider the permutation $(1 \hspace{0.5ex} 2 \hspace{0.5ex} 3)$.  If this permutation is associated with a binary word over $\{0,1 \}$, with $0<1$, it could only correspond to the word $00$.  On the other hand, if this permutation is associated with a word over 3 symbols, suppose $\{0,1,2 \}$ with $0<1<2$, then the permutation could be associated with any of $00$, $01$, $11$, or $12$.

For binary words the subpermutations depend on the order on the symbols used to compose $\w$, but the permutation complexity does not depend on the order.  For words over 3 or more symbols, not only do the subpermutations depend on the order on the alphabet but so does the permutation complexity.  For example, consider the Fibonacci word
$$t = 0100101001001010010100100101\ldots,$$
defined by iterating the morphism $0 \mapsto 01, 1 \mapsto 0$ on the letter $0$, and suppose the 1s are replaced by alternating $a$'s and $b$'s to create the word:
$$\hat{t} = 0a00b0a00b00a0b00a0b00a00b0a\ldots.$$
If the symbols in $\hat{t}$ are ordered $0<a<b$ there will be 5 distinct subpermutations of length 3, and if the symbols are ordered $a<0<b$ there will be only 4 distinct subpermutations of length 3.  The verification of this fact is left to the reader.

In view of the notion of an infinite permutation associated to an aperiodic word, it is natural to compute the permutation complexity of well-known classes of words.  In $\cite{Makar09}$, Makarov computes the permutation complexity of Sturmian words.  The goal of this paper is to determine the permutation complexity of the Thue-Morse word.

%
The Thue-Morse word, $T = T_0T_1T_2 \cdots$, is:
$$ T = 0110 1001 1001 0110 1001 0110 0110 1001 \cdots,$$
which can be generated by the morphism:
$$\mu_T:0 \mapsto 01, \hspace{1.5ex} 1 \mapsto 10, $$
by iterating on the letter $0$.  Axel Thue introduced this word in his studies of repetitions in words, and proved that the word $T$ is overlap-free ($\cite{Thue12}$).  A word $\w$ is said to be $\textit{overlap-free}$ if it does not contain a factor of the form $vuvuv$ for words $u$ and $v$, with $v$ non-empty.

The Thue-Morse word was again discovered independently by Marston Morse in 1921 $\cite{Morse21}$ through his study of differential geometry, and used in the foundations of symbolic dynamics.  For a more in depth look at further properties, independent discoveries, and applications of the Thue-Morse word see $\cite{AllSha99}$.

The factor complexity of the Thue-Morse word was computed independently by two groups in 1989, Brlek $\cite{Brlek89}$ and de Luca and Varricchio $\cite{LucaVarr89}$.  Our proof of the permutation complexity of the Thue-Morse word does not use the factor complexity function.

The permutation complexity of the Thue-Morse word can be found as follows.  For any $n \geq 2$, we can write $n$ as $n = 2^a + b$, with $0 < b \leq 2^a$.  Using this notation, it will shown that the formula for the permutation complexity of T, initially conjectured by M. Makarov, is
$$ \tau_T(n) = 2( 2^{a+1} + b - 2 ). $$

%
We give a a non-trivial proof of this formula here.  We start with some basic notation and definitions.  Some properties of infinite permutations are given in Section $\ref{GeneralPermResults}$.  The infinite permutation associated with the Thue-Morse word, $\pi_T$, is introduced in Section $\ref{ThueMorsePermutation}$.  Patterns found in the subpermutations of $\pi_T$ are studied in Section $\ref{TypeKandCompPairs}$, while Section $\ref{SecTypeOnePairs}$ investigates when a specific pattern occurs.  The formula for the permutation complexity is established in Section $\ref{FormulaForPermComp}$.  Low order subpermutations are listed in Appendix $\ref{SecTheSubperms}$ to be used as a base case for induction arguments.

\subsection{Words}
A $\textit{word}$ is a finite, (right) infinite, or bi-infinite sequence of symbols taken from a finite non-empty set, $A$, called an $\textit{alphabet}$.  The standard operation on words is concatenation, and is represented by juxtaposition of letters and words.  A $\textit{finite word}$ over $A$ is a word of the form $u = a_1 a_2 \ldots a_n$ with $n \geq 0$ (if $n=0$ we say $u$ is the $\textit{empty word}$, denoted $\epsilon$) and each $a_i \in A$; the $\textit{length}$ of the word $u$ is the number of symbols in the sequence and is denoted by $\abs{u} = n$.  For $a \in A$, let $\abs{u}_a$ denote the number of occurrences of the letter $a$ in the word $u$.  The set of all finite words over the alphabet $A$ is denoted by $A^*$, and is a free monoid with concatenation of words as the operation.  

%
A $\textit{(right) infinite word}$ over $A$ is a word of the form $\w = \w_0 \w_1 \w_2 \ldots$ with each $\w_i \in A$, and the set of all infinite words over $A$ is denoted $A^\N$.  Given $\w \in A^* \cup A^\N$, any word of the form $u=\w_i\w_{i+1} \ldots \w_{i+n-1}$, with $i \geq 0$, is called a $\textit{factor}$ of $\w$ of length $n \geq 1$.  The set of all factors of a word $\w$ is denoted by $\scr{F}(\w)$.  The set of all factors of length $n$ of $\w$  is denoted $\scr{F}_\w(n)$, and let $\rho_\w(n) = \abs{\scr{F}_\w (n)}$.  The function $\rho_\w: \N \rightarrow \N $ is called the $\textit{factor complexity function}$, or $\textit{subword complexity function}$, of $\w$ and it counts the number of factors of length $n$ of $\w$.  For a natural number $i$ we denote by $\w[i] = \w_i\w_{i+1}\w_{i+2}\w_{i+3}\ldots$ the $i$$\textit{-letter shift of}$ $\w$.  For natural numbers $i \leq j$, $\w[i,j] = \w_i\w_{i+1}\w_{i+2} \ldots \w_j$ denotes the factor of length $j-i+1$ starting at position $i$ in $\w$.

For words $u \in A^*$ and $v \in A^* \cup A^\N$ where $\w = uv$, we call $u$ a $\textit{prefix}$ of $\w$ and $v$ a $\textit{suffix}$ of $\w$.  A word $\w$ is said to be $\textit{periodic}$ of period $p$ if for each $i \in \N$, $\w_i = \w_{i+p}$, and $\w$ is said to be $\textit{eventually periodic}$ of period $p$ if there exists an $N \in \N$ so that for each $i > N$, $\w_i = \w_{i+p}$; or equivalently, $\w$ has a periodic suffix.  A word $\w$ is said to be $\textit{aperiodic}$ if it is not periodic or eventually periodic.

%
Let $A$ and $B$ be two finite alphabets.  A map $\varphi: A^* \rightarrow B^*$ so that $\varphi(uv) = \varphi(u)\varphi(v)$ for any $u,v \in A^*$ is called a $\textit{morphism}$ of $A^*$ into $B^*$, and $\varphi$ is defined by the image of each letter in $A$.  A morphism on $A$ is a morphism from $A^*$ into $A^*$, also called an $\textit{endomorphism}$ of $A$.  A morphism $\varphi$ is said to be $\textit{non-erasing}$ if the image of any non-empty word is not empty.

The action of a morphism $\varphi$ on $A$ can naturally be extended from $A^*$ to $A^\N$.  For any $\w = \w_0\w_1\w_2\ldots \in A^\N$, we define $\varphi(\w) = \varphi(\w_0)\varphi(\w_1)\varphi(w_2)\ldots$ as in the case for words in $A^*$.  We say that a word $\w$ is a $\textit{fixed point}$ of the morphism $\varphi$ if $\varphi(\w) = \w$.  If $\varphi$ is a morphism on A and if $\varphi(a) = au$ for some $a \in A$ and non-empty $u \in A^*$, $\varphi$ is said to be $\textit{prolongable}$ on $a$.  If $\varphi$ is a morphism on $A$ that is prolongable on some $a \in A$, then $\varphi^n(a)$ is a proper prefix of $\varphi^{n+1}(a)$ for each $n \in \N$.  The limit of the sequence $\left\{ \varphi^n(a) \right\}_{n \in \N}$ will be the unique infinite word
$$ \w = \lim_{n \rightarrow \infty} \varphi^n(a) = \varphi^\infty(a) = au\varphi(u)\varphi^2(u) \cdots $$
where $\w$ is a fixed point of $\varphi$, and we say that $\w$ is $\textit{generated}$ by $\varphi$.

\subsection{Permutations on words}
The idea of an infinite permutation that will be here used was introduced in $\cite{FlaFrid}$.  This paper will be dealing with permutation complexity of infinite words so the set used in the following definition will be $\N$ rather than an arbitrary countable set.  To define an $\textit{infinite permutation}$ $\pi$, start with a linear order $\prec_\pi$ on $\N$, together with the usual order $<$ on $\N$.  To be more specific, an infinite permutation is the ordered triple $\pi = \left\langle \N,\prec_\pi,<  \right\rangle$, where $\prec_\pi$ and $<$ are linear orders on $\N$.  The notation to be used here will be $\pi(i) < \pi(j)$ rather than $i \prec_\pi j.$

Given an infinite aperiodic word $\w = \w_0\w_1\w_2 \ldots$ on an alphabet $A$, fix a linear order on $A$.  We will use the binary alphabet $A = \{0, 1\}$ and use the natural ordering $0<1$.  Once a linear order is set on the alphabet, we can then define an order on the natural numbers based on the lexicographic order of shifts of $\w$.  Considering two shifts of $\w$ with $a \neq b$, $\w[a] = \w_a\w_{a+1}\w_{a+2} \ldots$ and $\w[b] = \w_b\w_{b+1}\w_{b+2} \ldots$, we know that $\w[a] \neq \w[b]$ since $\w$ is aperiodic.  Thus there exists some minimal number $c \geq 0$ so that $\w_{a+c} \neq \w_{b+c}$ and for each $0 \leq i < c$ we have $\w_{a+i} = \w_{b+i}$.  We call $\pi_\w$ the infinite permutation associated with $\w$ and say that $\pi_\w(a) < \pi_\w(b)$ if $\w_{a+c} < \w_{b+c}$, else we say that $\pi_\w(b) < \pi_\w(a)$.  

For natural numbers $a \leq b$ consider the factor $\w[a, b] = \w_a\w_{a+1} \ldots \w_b$ of $\w$ of length $b - a + 1$.  Denote the finite permutation of $\{ 1, 2, \ldots , b - a + 1 \}$ corresponding to the linear order by $\pi_\w[a,b]$.  That is $\pi_\w[a,b]$ is the permutation of $\{ 1, 2, \ldots , b - a + 1 \}$ so that for each $0 \leq i,j \leq (b - a)$, $ \pi_\w[a,b](i) < \pi_\w[a,b](j)$ if and only if $\pi_\w(a + i) < \pi_\w(a + j)$.  Say that $p = p_0p_1 \cdots p_n$ is a $\textit{(finite) subpermutation}$ of $\pi_\w$ if $p = \pi_\w[a,a+n]$ for some $a,n \geq 0$.  For the subpermutation $p = \pi_\w[a,a+n]$ of $\{1, 2, \cdots, n+1 \}$, we say the $\textit{length}$ of $p$ is $n+1$.

Denote the set of all subpermutations of $\pi_\w$ by $Perm_{\pi_\w}$, and for each positive integer $n$ let 
$$Perm_{\pi_\w}(n) = \{ \hspace{1.0ex} \pi_\w[i,i+n-1] \hspace{1.0ex} \left| \hspace{1.0ex} i \geq 0 \right. \hspace{1.0ex} \}$$
denote the set of distinct finite subpermutations of $\pi_\w$ of length $n$.  The $\textit{permutation complexity function}$ of $\w$ is defined as the total number of distinct subpermutations of $\pi_\w$ of a length $n$, denoted $\tau_\w(n) = \abs{Perm_{\pi_\w}(n)}$.

\begin{example}  
Let's consider the well-known Fibonacci word, 
$$t = 0100101001001010010100100101\ldots,$$
with the alphabet $A = \{0,1 \}$ ordered as $0 < 1$.  We can see that $t[2] = 001010\ldots $ is lexicographically less than $t[1] = 100101\ldots$, and thus $\pi_t(2) < \pi_t(1)$.

Then for a subpermutation, consider the factor $t[3,5] = 010$.  We see that $\pi_t[3,5] = (231)$ because in lexicographic order if we have $\pi_t(5) < \pi_t(3) < \pi_t(4)$.
\end{example}

\section{Some General Permutation Results}
\label{GeneralPermResults}

Initially work has been done with infinite binary words (see $\cite{Makar06,FlaFrid,Makar09,Makar09TM,Makar10,AvgFriKamSal}$).  Suppose $\w = \w_0\w_1\w_2\ldots$ is an aperiodic infinite word over the alphabet $A=\{ 0,1 \}$.  First let's look at some remarks about permutations generated by binary words where we use the natural order on $A$.
\begin{claim}
\emph{($\cite{Makar06}$)}
\label{PCClaim01}
For an infinite aperiodic word $\w$ over $A = \{ 0, 1 \}$ with the natural ordering we have:

(1) $\pi_\w(i) < \pi_\w(i+1)$ if and only if $\w_i = 0$.

(2) $\pi_\w(i) > \pi_\w(i+1)$ if and only if $\w_i = 1$.

(3) If $\w_i = \w_j$, then $\pi_\w(i) < \pi_\w(j)$ if and only if $\pi_\w(i+1) < \pi_\w(j+1)$
\end{claim}
\begin{lemma}
\emph{($\cite{Makar06}$)}
\label{PermComp01}
Given two infinite binary words u = $u_0u_1\ldots$ and $v=v_0v_1 \ldots$ with $\pi_u[0, n+1] = \pi_v[0, n+1]$, it follows that $u[0,n] = v[0,n]$.
\end{lemma}
We do have a trivial upper bound for $\tau_\w(n)$ being the number of permutations of length $n$, which is $n!$.  Lemma $\ref{PermComp01}$ directly implies a lower bound for the permutation complexity for a binary aperiodic word $\w$, namely the factor complexity of $\w$.  Thus, initial bounds on the permutation complexity can be seen to be:
$$ \rho_\w(n-1) \leq \tau_\w(n) \leq  n!$$

For $a \in A = \{0,1\}$, let $\bar{a}$ denote the $\textit{complement}$ of $a$, that is $\bar{0} = 1$ and $\bar{1} = 0$.  If $u = u_1u_2u_3 \cdots$ is a word over $A$, the $\textit{complement}$ of $u$ is defined to be the word composed of the complement of the letters in $u$, that is $\bar{u} = \bar{u}_1\bar{u}_2\bar{u}_3 \cdots$.  Let $\w$ be an infinite aperiodic binary word, we say the set of factors of $\w$ is $\textit{closed under complementation}$ if for each $u \in \scr{F}(\w)$ then $\bar{u} \in \scr{F}(\w)$.  The following lemma shows an interesting property of the subpermutations of the infinite permutation $\pi_\w$.
\begin{lemma}
\label{ClosedUnderCompliment}
Let $\w = \w_0\w_1\w_2\cdots$ be an infinite aperiodic binary word with factors closed under complementation.  If $p$ is a subpermutation of $\pi_\w$ of length $n$, then the subpermutation $q$ defined by $q_i = n-p_i +1$ for each $i$, is also a subpermutation of $\pi_\w$ of length $n$.
\end{lemma}
\begin{proof}
Let $p$ be a subpermutation of $\pi_\w$.  There is an $a \in \N$ so that $p = \pi_\w[a,a+n-1]$.  For each $i,j \in \{0,1, \ldots, n-1 \}$, if $p_i < p_j$ then $\w[a+i] < \w[a+j]$ and there is some finite word $u_{i,j}$ so that 
\begin{align*}
\w[a+i] &= u_{i,j}0\cdots \\
\w[a+j] &= u_{i,j}1\cdots 
\end{align*}

Let $v$ be the prefix of $\w[a]$ so that for each $i,j \in \{ 0,1, \ldots, n-1 \}$, $v$ contains both $u_{i,j}0$ and $u_{i,j}1$.  Since the set of factors of $\w$ is closed under complementation, $\bar{v}$ is a factor of $\w$.  There is a $b$ so that $\bar{v}$ is a prefix of $\w[b]$, and let $q = \pi_\w[b,b+n-1]$.  For each $i,j \in \{0, 1, \ldots, n-1 \}$, if $p_i < p_j$ 
\begin{align*}
\w[b+i] &= \bar{u}_{i,j}1\cdots \\
\w[b+j] &= \bar{u}_{i,j}0\cdots
\end{align*}
and thus, $q_i > q_j$.

For any $i \in \{0,1,  \ldots, n-1 \}$ there are $p_i - 1$ many $j$ so that $p_j < p_i$ and there are $n - p_i$ many $j$ so that $p_j > p_i$.  Therefore there are $n - p_i$ many $j$ so that $q_j < q_i$, so $q_i = n - p_i + 1$.
$\qed$
\end{proof}

\begin{definition}
\label{SameForm}
Two permutations $p$ and $q$ of $\{1, 2, \ldots, n  \}$ have the $\textit{same form}$ if for each $i = 0, 1, \ldots, n-1$, $p_i < p_{i+1}$ if and only if $q_i < q_{i+1}$.  For a binary word $u$ of length $n-1$, say that $p$ $\textit{has form u}$ if 
$$p_i<p_{i+1} \Longleftrightarrow u_i = 0$$
for each $i = 0, 1, \ldots, n-2$.
\end{definition}

\section{The Thue-Morse Permutation}
\label{ThueMorsePermutation}

In this section the action of the Thue-Morse morphism on the subpermutations of $\pi_T$ will be investigated.  This action will induce a well-defined map on the subpermutations of $\pi_T$ and lead to an initial upper-bound on the permutation  complexity of $T$.

The Thue-Morse word is:
$$ T = 0110 1001 1001 0110 1001 0110 0110 1001 \cdots,$$
and the Thue-Morse morphism is:
$$\mu_T:0 \rightarrow 01, \hspace{1.5ex} 1 \rightarrow 10. $$
It can readily be verified that if $a$ is a natural number then 
$$\mu_T(T[a]) = T[2a]$$
since for any letter $x \in \{0,1 \}$, $\abs{\mu_T(x)}=2$.  

A nice property of the factors of $T$ is that any factor of length 5 or greater contains either $00$ or $11$.  Another interesting property is that for any $i \in \N$, $T[2i,2i+1]$ will be either 01 or 10.  Thus any occurrence of $00$ or $11$ must be a factor of the form $T[2i+1,2i+2]$ for some $i \in \N$.  Therefore any factors $T[2i,2i+n]$ and $T[2j+1,2j+1+n]$ where $n \geq 4$ cannot be equal based on the location of the factors $00$ or $11$.

Let $\pi_T$ be the infinite permutation associated to the Thue-Morse word $T$.  For notational purposes, the set of all subpermutations of $\pi_T$ of length $n$ will be denoted as $Perm(n)$.  


Let $a$ and $n$ be natural numbers and suppose we want to determine if $T[a] < T[a+n]$.  There will be some (possibly empty) factor $u$ of $T$, and suffixes $x$ and $y$ of $T$ so that $T[a] = u\lambda x$ and $T[a+n] = u \bar{\lambda}y$, for $\lambda \in \{0,1 \}$.  If $\abs{u} \geq n+1$ we would have $T_{a+i} = T_{a+n+i}$ for each $i = 0, 1, \ldots, n$, and thus $T[a,a+n] = T[a+n,a+2n]$, and $T[a,a+2n]$ would violate the fact that $T$ is overlap-free.  Thus $\abs{u} \leq n$, and if $\abs{u} = n$ we have $T[a,a+n-1] = T[a+n,a+2n-1]$ and $T_{2n} = \overline{T_a}$.  Therefore the subpermutation $\pi_T[a,a+n]$ can be determined within the factor $T[a,a+2n]$ of length $2n+1$.  Thus the trivial bounds for the permutation complexity of the Thue-Morse word $T$ are
$$ \rho_T(n-1) \leq \tau_T(n) \leq \rho_T(2n-1). $$

Since the factor complexity of the Thue-Morse word is known (see $\cite{Brlek89,LucaVarr89}$) we can find all factors of a given length.  Thus for any natural number $n$, all factors of $T$ of length $2n-1$ can be identified and thus the set of all subpermutations of $\pi_T$ of length $n$, $Perm(n)$, can be identified as well.  The subpermutations of $\{1, 2, \ldots, n\}$ have been identified for relatively low $n$ (up to $n=65$) and in these cases no more than two subpermutations of any length were identified to have the same form.  In other words, for any factor $u$ of $T$ of length $n \leq 64$ there are at most two subpermutations of length $n+1$ having form $u$.  

This section will deal with some properties of $\pi_T$.  Something to note about the Thue-Morse morphism is that it is an order preserving morphism, as shown by the following lemma.
\begin{lemma}
\label{OrderPresMorph}
For natural numbers $a$ and $b$, $T[a] < T[b]$ if and only if $\mu_T(T[a]) < \mu_T(T[b])$.
\end{lemma}
\begin{proof}
If $T[a] < T[b]$, then there exists a finite factor $u$ of $T$, and suffixes $x$ and $y$ of $T$ so that
\begin{align*}
T[a] &= u0x \\
T[b] &= u1y.
\end{align*}
Thus we can see
\begin{align*}
\mu_T(T[a]) &= \mu_T(u)01\mu_T(x) \\
\mu_T(T[b]) &= \mu_T(u)10\mu_T(y)
\end{align*}
and therefore $\mu_T(T[a]) < \mu_T(T[b])$.

Suppose $\mu_T(T[a]) < \mu_T(T[b])$, then there exists a finite factor $u$ of $T$, and suffixes $x$ and $y$ of $T$ so that
\begin{align*}
\mu_T(T[a]) &= u0x \\
\mu_T(T[b]) &= u1y 
\end{align*}
If $u$ ends with a $0$, then $\mu_T(T[a])$ would have 00 at the end of $u0$, so $u$ ends with $10$ and $0x$ starts with $01$.  If $u$ ends with a $1$, then $\mu_T(T[b])$ would have 11 at the end of $u1$, so $u$ ends with $01$ and $1x$ starts with $10$.  In either case we have there is some factor $v$ so that $\mu_T(v) = u$.  Hence a prefix of $\mu_T(T[a])$ is $\mu_T(v)01$ and a prefix of $\mu_T(T[b])$ is $\mu_T(v)10$

Thus a prefix of $T[a]$ is $v0$ and a prefix of $T[b]$ is $v1$.  Therefore $T[a] < T[b]$.
$\qed$
\end{proof}

\begin{lemma}
\label{ImgOfZerosAndOnesTM}
If $u$ and $v$ are shifts of $T$ so that for some $a$ and $b$ $u = 0T[a]$ and $v = 1T[b]$, and hence $u<v$, $\mu_T(u) = 01\mu_T(T[a])$, and $\mu_T(v) = 10\mu_T(T[b])$.  Thus $0\mu_T(T[b]) < 01\mu_T(T[a]) < 10\mu_T(T[b]) < 1\mu_T(T[a])$.
\end{lemma}
\begin{proof}
The first letters in $T[a]$ will be either $01$ or $1$, thus $\mu_T(T[a])$ will start with either $0110$ or $10$, respectively.  The first letters in $T[b]$ will be either $10$ or $0$, thus $\mu_T(T[b])$ will start with either $1001$ or $01$, respectively.

Then $0\mu_T(T[b])$ will start with $01001$ or $001$ and $01\mu_T(T[a])$ will start with $010110$ or $0110$.  Thus $001<01001<010110<0110$, so 
$$0\mu_T(T[b]) < 01\mu_T(T[a]).$$

Then $10\mu_T(T[b])$ will start with $101001$ or $1001$ and $1\mu_T(T[a])$ will start with $10110$ or $110$.  Thus $1001<101001<10110<110$, so 
$$10\mu_T(T[b]) < 1\mu_T(T[a]).$$

Therefore 
$$0\mu_T(T[b]) < 01\mu_T(T[a]) < 10\mu_T(T[b]) < 1\mu_T(T[a]).$$
$\qed$
\end{proof}

Let $u$ be a factor of $T$ of length $n$.  There is an $a \in \N$ so that $u = T[a,a+n-1]$.  Also recall that $\abs{u}_1$ is the number of occurrences of the letter $1$ in $u$, and that $\abs{u}_1 = n-\abs{u}_0$.  Let $p = \pi_T[a,a+n]$ be a subpermutation of $\pi_T$ with form $u$.  Then $\mu_T(u) = T[2a, 2a + 2n-1]$, and let $p'$ be the subpermutation $p' = \pi_T[2a, 2a + 2n]$ with form $\mu_T(u)$.  When Lemma $\ref{ImgOfZerosAndOnesTM}$ is used with this notation, for $0\leq i,j \leq n-1$, where $T_i=0$ and $T_j = 1$, we have $p_i < p_j$ and $p'_{2j+1} < p'_{2i} < p'_{2j} < p'_{2i+1}$.  The following lemma describes the values of $p'$ in terms of the values of $p$.
\begin{proposition}
\label{CalculateTheFwdImage}
Let $u$, $p$, and $p'$ be as described above.  For any $i \in \{0, 1, \ldots, n \}$:
$$ p'_{2i} = p_i + \abs{u}_1 $$
and for any $i \in \{0, 1, \ldots, n-1 \}$:
$$ p'_{2i+1} = \begin{cases} 
p_i + \abs{u}_1+(n+1) & \text{if $p_i < p_{i+1}$ and $p_i < p_n$} \\
p_i + \abs{u}_1+n & \text{if $p_i < p_{i+1}$ and $p_i > p_n$} \\
p_i + \abs{u}_1 - n & \text{if $p_i > p_{i+1}$ and $p_i < p_n$} \\
p_i + \abs{u}_1 - (n+1) & \text{if $p_i > p_{i+1}$ and $p_i > p_n$}
\end{cases} $$
\end{proposition}
\begin{proof}
To take care of the $p'_{2i}$ terms, let $i \in \{0, 1, \ldots, n \}$.  There will be $p_i-1$ many $j$ so that $p_i > p_j$, so there are $p_i-1$ many $j$ so that $p'_{2i} > p'_{2j}$.  Clearly, if $p_i  < p_j$ then $p'_{2i} < p'_{2j}$.  So there are exactly $p_i-1$ many even $j$ so that $p'_{2i} > p'_j$.  There are $\abs{u}_1$ many $j$ so that $T_{a+j} = 1$, so there are $\abs{u}_1$ many $j$ so that $p'_{2i} > p'_{2j+1}$ and $\abs{u}_0$ many $j$ so that $T_{a+j}=0$, so $p'_{2i} < p'_{2j+1}$.  So there are exactly $\abs{u}_1$ many odd $j$ so that $p'_{2i} > p'_j$.  Thus there are exactly $p_i-1+\abs{u}_1$ many $j$ so that $p'_{2i} > p'_j$, and therefore $ p'_{2i} =(p_i-1+\abs{u}_1)+1 = p_i + \abs{u}_1 $.

The $p'_{2i+1}$ terms will be done in two cases.  First when $p_i < p_{i+1}$ and then when $p_i > p_{i+1}$.

$\textbf{Case a:}$ Suppose that $p_i < p_{i+1}$, so $T_{a+i} = 0$.  For each $j = 0, 1, \ldots, n$ we must have $p'_{2i+1} > p'_{2j}$, so for each even $j$ (there are $n+1$ many such $j$) $p'_{2i+1} > p'_j$.  There are $\abs{u}_1$ many $j$ so that $T_{a+j} = 1$, so there are $\abs{u}_1$ many $j$ so that $p_{2i+1} > p_{2j+1}$.  Thus the only other $j$ where $p'_{2j+1}$ can be less than $p'_{2i+1}$ are $j \in \{0, 1, \ldots, n-1  \}$ where $T_{a+j} = 0$ and $p_i > p_j$.

$\textbf{Subcase a.1:}$ If $p_i < p_n$ then there are $p_i - 1$ many $j$ so that $T_{a+j}=0$ and $p_i > p_j$, and then $n-p_i-\abs{u}_1= \abs{u}_0 - p_i$ many $j$ so that $T_{a+j}=0$ and $p_i < p_j$.  Thus there can only be $(n+1) + \abs{u}_1 + p_i - 1$ many $j$ so that $p'_{2i+1} > p'_j$, and therefore $p'_{2i+1} = (n+1) + \abs{u}_1 + p_i - 1 + 1 = p_i + \abs{u}_1 + (n+1)$.

$\textbf{Subcase a.2:}$ If $p_i > p_n$ then there are $p_i - 2$ many $j$ so that $T_{a+j}=0$ and $p_i > p_j$ (since $T_{a+n}$ is not in $u = T[a,a+n-1]$), and then $n-(p_i-1)-\abs{u}_1= \abs{u}_0 - (p_i - 1)$ many $j$ so that $T_{a+j}=0$ and $p_i < p_j$.  Thus there can only be $(n+1) + \abs{u}_1 + p_i - 2$ many $j$ so that $p'_{2i+1} > p'_j$, and therefore $p'_{2i+1} = (n+1) + \abs{u}_1 + p_i - 2 + 1 = p_i + \abs{u}_1 + n$.

$\textbf{Case b:}$ Suppose that $p_i > p_{i+1}$, so $T_{a+i} = 1$.  For each $j = 0, 1, \ldots, n$ we must have $p'_{2i+1} < p'_{2j}$, so for each even $j$ (there are $n+1$ many such $j$) $p'_{2i+1} < p'_j$.  There are $\abs{u}_0$ many $j$ so that $T_{a+j} = 0$, so there are $\abs{u}_0$ many $j$ so that $p_{2i+1} < p_{2j+1}$.  Thus the only other $j$ where $p'_{2j+1}$ can be less than $p'_{2i+1}$ are $j \in \{0, 1, \ldots, n-1  \}$ where $T_{a+j} = 1$ and $p_i > p_j$.

$\textbf{Subcase b.1:}$ If $p_i < p_n$ then there are $(p_i - 1) - \abs{u}_0$ many $j$ so that $T_{a+j}=1$ and $p_i > p_j$, and there can only be $\abs{u}_1 - (p_i - 1 - \abs{u}_0) - 1 = n-p_i$ many $j$ so that $T_{a+j}=1$ and $p_i < p_j$ (since $T_{a+n}$ is not in $u = T[a,a+n-1]$).  Thus there can only be $(p_i - 1) - \abs{u}_0 = p_i - 1 - (n-\abs{u}_1) = p_i +\abs{u}_1 - n -1$ many $j$ so that $p'_{2i+1} > p'_j$, and therefore $p'_{2i+1} = p_i +\abs{u}_1 - n -1 + 1 = p_i + \abs{u}_1 - n$.

$\textbf{Subcase b.2:}$ If $p_i > p_n$ then there are $(p_i - 2) - \abs{u}_0$ many $j$ so that $T_{a+j}=1$ and $p_i > p_j$ (since $T_{a+n}$ is not in $u = T[a,a+n-1]$), and there can only be $\abs{u}_1 - (p_i - 2 - \abs{u}_0) - 1 = (n+1)-p_i$ many $j$ so that $T_{a+j}=1$ and $p_i < p_j$.  Thus there can only be $(p_i - 2) - \abs{u}_0 = p_i - 2 - (n-\abs{u}_1) = p_i +\abs{u}_1 - n -2$ many $j$ so that $p'_{2i+1} > p'_j$, and therefore $p'_{2i+1} = p_i +\abs{u}_1 - n -2 + 1 = p_i + \abs{u}_1 - (n+1)$.
$\qed$
\end{proof}

Fix a subpermutation $p=\pi_T[a,a+n]$, and then let $p'=\pi_T[2a,2a+2n]$.  So the terms of $p'$ can be defined using the method defined in Proposition $\ref{CalculateTheFwdImage}$.  Let $q=\pi_T[b,b+n]$, $b \neq a$, be a subpermutation of $\pi_T$ and let $q'=\pi_T[2b,2b+2n]$ as in Proposition $\ref{CalculateTheFwdImage}$.  The following lemma concerns the relationship of $p$ and $q$ to $p'$ and $q'$.  Therefore the idea of $p'$ can ne used to define a map on the subpermutations of $\pi_T$, and the map will be well-defined by Proposition $\ref{CalculateTheFwdImage}$.

\begin{lemma}
\label{pISqIFFppISqp}
$p \neq q$ if and only if $p' \neq q'$.
\end{lemma}
\begin{proof}
Supposing that $p \neq q$, there are $i,j \in \{0, 1, \ldots,  n \}$ so that $p_i < p_j$ and $q_i > q_j$ and thus 
\begin{align*}
T[a+i] &< T[a+j] \\
T[b+i] &> T[b+i].
\end{align*}
Then since the Thue-Morse morphism is order preserving we have
\begin{align*}
T[2(a+i)] = \mu_T(T[a+i]) &< \mu_T(T[a+j]) = T[2(a+j)] \\
T[2(b+i)] = \mu_T(T[b+i]) &> \mu_T(T[b+j]) = T[2(b+j)].
\end{align*}
Therefore $p'_{2(a+i)}<p'_{2(a+j)}$ and $q'_{2(b+i)}>q'_{2(b+j)}$ so $p' \neq q'$.

Now to show by contrapositive, suppose that $p=q$, so $p_i = q_i$ for each $i \in \{0, 1, \ldots,  n \}$.  Since $p=q$, $p$ and $q$ have the same form, because $p_i < p_{i+1}$ if and only if $q_i < q_{i+1}$, so $T[a,a+n-1] = T[b,b+n-1]$ and thus $T[2a,2a+2n-1] = T[2b,2b+2n-1]$.  Then by Proposition $\ref{CalculateTheFwdImage}$ it should be clear that for each $j \in \{0, 1, \ldots,  2n \}$ we have $p'_j = q'_j$, and thus $p' = q'$.

Therefore if $p' \neq q'$ then $p \neq q$.
$\qed$
\end{proof}

\begin{corollary}
\label{CorTo_pisqiff}
If $p=\pi_T[a,a+n]=\pi_T[b,b+n]$ for some $a \neq b$, then $\pi_T[2a,2a+2n]=\pi_t[2b,2b+2n]$. 
\end{corollary}

Thus there is a well-defined function on the subpermutations of $\pi_T$.  Let $p = \pi_T[a,a+n]$, and define $\phi(p) = p' = \pi_T[2a,2a+2n]$ using the formula in Proposition $\ref{CalculateTheFwdImage}$.  Thus we have the map 
$$\phi:Perm(n+1) \rightarrow Perm(2n+1)$$
which is injective by Lemma $\ref{pISqIFFppISqp}$.  Not all subpermutations of $\pi_T$ will be the image under $\phi$ of another subpermutation.

Let $n \geq 5$ and $a$ be natural numbers.  Then $n$ and $a$ can be either even or odd, and for the subpermutation $\pi_T[a,a+n]$, there exist natural numbers $b$ and $m$ so that one of 4 cases hold:
\begin{enumerate}
\item $\pi_T[a,a+n] = \pi_T[2b,2b+2m]$, even starting position with odd length
\item $\pi_T[a,a+n] = \pi_T[2b,2b+2m-1]$, even starting position with even length
\item $\pi_T[a,a+n] = \pi_T[2b+1,2b+2m]$, odd starting position with even length
\item $\pi_T[a,a+n] = \pi_T[2b+1,2b+2m+1]$, odd starting position with odd length
\end{enumerate}

Consider two subpermutations of length $n > 5$, $\pi_T[2c, 2c+n]$ and $\pi_T[2d+1, 2d+n+1]$.  The subpermutations $\pi_T[2c, 2c+n]$ will have form $T[2c, 2c+n-1]$, and $\pi_T[2d+1, 2d+n+1]$ will have form $T[2d+1, 2d+n]$.  Since the length of these factors is at least 5, we know that $T[2c, 2c+n-1] \neq T[2d+1, 2d+n]$, and thus $\pi_T[2c, 2c+n] \neq \pi_T[2d+1, 2d+n+1]$ because they do not have the same form.  Thus we can break up the set $Perm(n)$ into two classes of subpermutations, namely the subpermutations that start at an even position or an odd position.  So say that $Perm_{ev}(n)$ is the set of subpermutations $p$ of length $n$ so that $p = \pi_T[2b,2b+n-1]$ for some $b$, and that $Perm_{odd}(n)$ is the set of subpermutations $p$ of length $n$ so that $p = \pi_T[2b+1,2b+n]$ for some $b$.  Thus
$$Perm(n) = Perm_{ev}(n) \cup Perm_{odd}(n),$$
where we have  
$$Perm_{ev}(n) \cap Perm_{odd}(n) = \emptyset.$$

Thus for $n \geq 3$, $Perm_{ev}(2n+1)$ is the set of all subpermutations of length $2n+1$ starting at an even position.  So for $\pi_T[2a,2a+2n]$, we know there is a subpermutation $p = \pi_T[a,a+n]$ so that $\phi(p) = p' = \pi_T[2a,2a+2n]$.  Thus the map 
$$\phi:Perm(n+1) \rightarrow Perm_{ev}(2n+1)$$
is also a surjective map, and is thus a bijection.  The next definition about the restriction of subpermutations will be helpful to count the size of the sets $Perm_{odd}(2n)$, $Perm_{ev}(2n)$, and $Perm_{odd}(2n+1)$.

\begin{definition}
Let $p = \pi[a,a+n]$ be a subpermutation of the infinite permutation $\pi$.  The $\textit{left restriction of}$ $p$, denoted by $L(p)$, is the subpermutation of $p$ so that $L(p) = \pi[a, a+n-1]$.  The $\textit{right restriction of}$ $p$, denoted by $R(p)$, is the subpermutation of $p$ so that $R(p) = \pi[a+1, a+n]$.  The $\textit{middle restriction of}$ $p$, denoted by $M(p)$, is the subpermutation of $p$ so that $M(p) = R(L(p)) = L(R(p)) = \pi[a+1, a+n-1]$.
\end{definition}

For each $i$, there are $p_i-1$ terms in $p$ that are less than $p_i$ and there are $n-p_i$ terms that are greater than $p_i$.  Thus consider $i \in \{0, 1, \ldots, n-1\}$ and the values of $L(p)_i$ and $R(p)_i$.  If $p_0 < p_{i+1}$ there will be $p_{i+1}-2$ terms in $R(p)$ less than $R(p)_i$ so we have $R(p)_i = p_{i+1}-1$.  In a similar sense, if $p_n < p_i$ we have $L(p)_i = p_i - 1$.  If $p_0 > p_{i+1}$ there will be $p_{i+1}-1$ terms in $R(p)$ less than $R(p)_i$ so we have $R(p)_i = p_{i+1}$.  In a similar sense, if $p_n > p_i$ we have $L(p)_i = p_i $.  

The values in $M(p)$ can be found by finding the values in $R(L(p))$ or $L(R(p))$.  Since $R(L(p))$ or $L(R(p))$ correspond to the same subpermutation of $p$, $R(L(p))_i < R(L(p))_j$ if and only if $L(R(p))_i < L(R(p))_j$.  Therefore $R(L(p)) = L(R(p))$.

It should also be clear that if there are two subpermutations $p= \pi_T[a,a+n]$ and $q = \pi_T[b,b+n]$ so that $p=q$ then $L(p) = L(q)$, $R(p) = R(q)$, and $M(p) = M(q)$ since if $p=q$ then $p_i < p_j$ if and only if $q_i < q_j$.  

For $p=\pi_T[a,a+n]$, we can then define three additional maps by looking at the left, right, and middle restrictions of $\phi(p) = p'$.  These maps are
\begin{align*}
\phi_L:Perm(n+1) &\rightarrow Perm_{ev}(2n) \\
\phi_R:Perm(n+1) &\rightarrow Perm_{odd}(2n) \\
\phi_M:Perm(n+2) &\rightarrow Perm_{odd}(2n+1) 
\end{align*}
and are defined by
\begin{align*}
\phi_L(p) &= L(\phi(p)) = L(p')\\
\phi_R(p) &= R(\phi(p)) = R(p')\\
\phi_M(p) &= M(\phi(p)) = M(p')
\end{align*}
It can be readily verified that these three maps are surjective.  To see an example of this, consider the map $\phi_L$, and let $\pi_T[2b,2b+2n-1]$ be a subpermutation in $Perm_{ev}(2n)$.  Then for the subpermutation $p=\pi_T[b,b+n]$, $\phi_L(p) = L(p') = \pi_T[2b,2b+2n-1]$ so $\phi_L$ is surjective.  A similar argument will show that $\phi_R$ and $\phi_M$ are also surjective.  

\begin{lemma}
\label{UpperBoundForTau}
For $n \geq 2$:
\begin{align*}
\tau_T(2n) &\leq 2(\tau_T(n+1)) \\
\tau_T(2n+1) &\leq  \tau_T(n+1) + \tau_T(n+2) 
\end{align*}
\end{lemma}
\begin{proof}
Let $n \geq 2$.  We have:
\begin{align*}
\abs{Perm_{ev}(2n)} &\leq \abs{Perm(n+1)} \\
\abs{Perm_{odd}(2n)} &\leq \abs{Perm(n+1)} \\
\\
\abs{Perm_{ev}(2n+1)} &= \abs{Perm(n+1)} \\
\abs{Perm_{odd}(2n+1)} &\leq \abs{Perm(n+2)}
\end{align*}
since $\phi$ is a bijection, and the 3 maps $\phi_L$, $\phi_R$, and $\phi_M$ are all surjective.  Thus we have the following inequalities:
\begin{align*}
\tau_T(2n) &= \abs{Perm(2n)} = \abs{Perm_{ev}(2n)} + \abs{Perm_{odd}(2n)} \\
 &\leq \abs{Perm(n+1)} + \abs{Perm(n+1)} = 2(\tau_T(n+1)) \\
\\
\tau_T(2n+1) &= \abs{Perm(2n+1)} = \abs{Perm_{ev}(2n+1)} + \abs{Perm_{odd}(2n+1)} \\
 &\leq \abs{Perm(n+1)} + \abs{Perm(n+2)} = \tau_T(n+1) + \tau_T(n+2) 
\end{align*}
$\qed$
\end{proof}

The three maps $\phi_L$, $\phi_R$, and $\phi_M$ are not injective maps.  To see this, consider the subpermutations 
\begin{align*}
&p=\pi_T[5,9] = [2 \hspace{.5ex} 3 \hspace{.5ex} 5 \hspace{.5ex} 4 \hspace{.5ex} 1] \\
&q=\pi_T[23,27] = [1 \hspace{.5ex} 3 \hspace{.5ex} 5 \hspace{.5ex} 4 \hspace{.5ex} 2].
\end{align*}
Both of these subpermutations have form $T[5,8] = T[23,26] = 0011$.  Then applying the maps we see:
\begin{align*}
&p' = \phi(p) = \pi_T[10,18] = [4 \hspace{.5ex} 8 \hspace{.5ex} 5 \hspace{.5ex} 9 \hspace{.5ex} 7 \hspace{.5ex} 2 \hspace{.5ex} 6 \hspace{.5ex} 1 \hspace{.5ex} 3] \\
&q' = \phi(q) = \pi_T[46,54] = [3 \hspace{.5ex} 8 \hspace{.5ex} 5 \hspace{.5ex} 9 \hspace{.5ex} 7 \hspace{.5ex} 2 \hspace{.5ex} 6 \hspace{.5ex} 1 \hspace{.5ex} 4] 
\end{align*}
\begin{align*}
&\phi_L(p) = \pi_T[10,17] = [3 \hspace{.5ex} 7 \hspace{.5ex} 4 \hspace{.5ex} 8 \hspace{.5ex} 6 \hspace{.5ex} 2 \hspace{.5ex} 5 \hspace{.5ex} 1] \\
&\phi_L(q) = \pi_T[46,53] = [3 \hspace{.5ex} 7 \hspace{.5ex} 4 \hspace{.5ex} 8 \hspace{.5ex} 6 \hspace{.5ex} 2 \hspace{.5ex} 5 \hspace{.5ex} 1] 
\end{align*}
\begin{align*}
&\phi_R(p) = \pi_T[11,18] = [7 \hspace{.5ex} 4 \hspace{.5ex} 8 \hspace{.5ex} 6 \hspace{.5ex} 2 \hspace{.5ex} 5 \hspace{.5ex} 1 \hspace{.5ex} 3] \\
&\phi_R(q) = \pi_T[47,54] = [7 \hspace{.5ex} 4 \hspace{.5ex} 8 \hspace{.5ex} 6 \hspace{.5ex} 2 \hspace{.5ex} 5 \hspace{.5ex} 1 \hspace{.5ex} 3] 
\end{align*}
\begin{align*}
&\phi_M(p) = \pi_T[11,17] = [6 \hspace{.5ex} 3 \hspace{.5ex} 7 \hspace{.5ex} 5 \hspace{.5ex} 2 \hspace{.5ex} 4 \hspace{.5ex} 1] \\
&\phi_M(q) = \pi_T[47,53] = [6 \hspace{.5ex} 3 \hspace{.5ex} 7 \hspace{.5ex} 5 \hspace{.5ex} 2 \hspace{.5ex} 4 \hspace{.5ex} 1]
\end{align*}
So $p' \neq q'$ but $\phi_L(p) = \phi_L(q)$, $\phi_R(p) = \phi_R(q)$, and $\phi_M(p) = \phi_M(q)$, and these maps are not injective in general.  Hence the values in Lemma $\ref{UpperBoundForTau}$ are only an upper bound.  The next goal is to determine when these maps are not injective.

\section{Type $k$ and Complementary Pairs}
\label{TypeKandCompPairs}

An interesting pattern occurs in some subpermutations of $\pi_T$.  The subpermutations that follow this pattern are said to be subpermutations of type $k$ which is described in the next definition.  Proposition $\ref{CalculateTheFwdImage}$ will be used inductively to show the maps $\phi$, $\phi_L$, $\phi_R$, and $\phi_M$ preserve subpermutations of type $k$.  An induction argument with this fact will be used to show that two subpermutations have the same form if and only if they are a complimentary pair of type $k$, defined below.  A corollary of this will determine when the maps $\phi_L$, $\phi_R$, and $\phi_M$ are bijective.

\begin{definition}
A subpermutation $p = \pi_T[a,a+n]$ is of $\textit{type $k$}$, for $k \geq 1$, if $p$ can be decomposed as
$$ p = [\alpha_1 \cdots \alpha_k \lambda_1 \cdots \lambda_l \beta_1 \cdots \beta_k] $$
where $\alpha_i = \beta_i + \varepsilon$ for each $i = 1, 2, \ldots, k$ and an $\varepsilon \in \{ -1, 1 \}$.
\end{definition}

Some examples of subpermutations of type $1$, 2, and 3 (respectively) are:
\begin{align*}
&\pi_T[5,9] = [2 \hspace{.5ex} 3 \hspace{.5ex} 5 \hspace{.5ex} 4 \hspace{.5ex} 1] \\
&\pi_T[20,25] =  [2 \hspace{.5ex} 5 \hspace{.5ex} 4 \hspace{.5ex} 1 \hspace{.5ex} 3 \hspace{.5ex} 6] \\
&\pi_T[6,12] =  [3 \hspace{.5ex} 7 \hspace{.5ex} 5 \hspace{.5ex} 1 \hspace{.5ex} 2 \hspace{.5ex} 6 \hspace{.5ex} 4]
\end{align*}

\begin{definition}
Suppose that the subpermutation $p = \pi_T[a,a+n]$ is of type $k$ so that for $\varepsilon \in \{-1, 1 \}$, $\alpha_i = \beta_i + \varepsilon$ for each $i = 1, 2, \ldots, k$.  If there exists a subpermutation $q = \pi_T[b,b+n]$ of type $k$ so that $p$ and $q$ can be decomposed as:
\begin{align*}
p &= \pi_T[a,a+n] = [\alpha_1 \cdots \alpha_k \lambda_1 \cdots \lambda_l \beta_1 \cdots \beta_k] \\ 
q &= \pi_T[b,b+n] = [\beta_1 \cdots \beta_k \lambda_1 \cdots \lambda_l \alpha_1 \cdots \alpha_k]  
\end{align*}
then $p$ and $q$ are said to be a $\textit{complementary pair of type $k$}$.  If $p$ and $q$ are a $\textit{complementary pair of type}$ $k \leq 0$ then $p = q$.
\end{definition}

The subpermutations 
\begin{align*}
&\pi_T[5,9] = [2 \hspace{.5ex} 3 \hspace{.5ex} 5 \hspace{.5ex} 4 \hspace{.5ex} 1] \\
&\pi_T[23,27] = [1 \hspace{.5ex} 3 \hspace{.5ex} 5 \hspace{.5ex} 4 \hspace{.5ex} 2]
\end{align*}
are a complementary pair of type 1. The following subpermutation of type 1
$$ \pi_T[0,3] =  [2 \hspace{.5ex} 4 \hspace{.5ex} 3 \hspace{.5ex} 1] $$
does not have a complementary pair, since $[1 \hspace{.5ex} 4 \hspace{.5ex} 3 \hspace{.5ex} 2] $ is not a subpermutation of $\pi_T$.

The following proposition considers subpermutations of type $k$, and complementary pairs of type $k$.

\begin{proposition}
\label{ImageOfTypeK}
Suppose $p = \pi_T[a,a+n]$ is of type $k$ and $q = \pi_T[b,b+n]$ is of type $k$, with $k \geq 1$, and that $p$ and $q$ are a complementary pair of type $k$.
\begin{itemize}
\item[(a)] $\phi(p)$ is of type $2k-1$, and if $k \geq 2$ then $\phi_L(p)$ and $\phi_R(p)$ are of type $2k-2$ and $\phi_M(p)$ is of type $2k-3$.
\item[(b)] $\phi(p)$ and $\phi(q)$ are a complementary pair of type $2k-1$.
\item[(c)] $\phi_L(p)$ and $\phi_L(q)$ are a complementary pair of type $2k-2$.
\item[(d)] $\phi_R(p)$ and $\phi_R(q)$ are a complementary pair of type $2k-2$.
\item[(e)] $\phi_M(p)$ and $\phi_M(q)$ are a complementary pair of type $2k-3$.
\end{itemize}
\end{proposition}
\begin{proof}
Since $p$ and $q$ are a complementary pair of type $k$ they can be decomposed as
\begin{align*}
p &= \pi_T[a,a+n] = [\alpha_1 \cdots \alpha_k \lambda_1 \cdots \lambda_l \beta_1 \cdots \beta_k] \\ 
q &= \pi_T[b,b+n] = [\beta_1 \cdots \beta_k \lambda_1 \cdots \lambda_l \alpha_1 \cdots \alpha_k]  
\end{align*}
and for $\varepsilon \in \{-1, 1 \}$, $\alpha_i = \beta_i + \varepsilon$ for each $i = 1, 2, \ldots, k$.  For the values of $k$ and $l$, $2k+l = n+1$ and $4k+2l-1=2n+1$.

$\vspace{1.0ex}$

$\textbf{(a)}$  The first thing to show is that $\phi(p)$ is of type $2k-1$.

For $i \in \{0, 1, \ldots, k-1  \}$ we have $p_i = p_{n-(k-1)+i}+\varepsilon$, so by Proposition $\ref{CalculateTheFwdImage}$:
$$p'_{2i} = p'_{2(n-(k-1)+i)} +\varepsilon $$
For $i \in \{0, 1, \ldots, k-2 \}$, $p_i < p_{i+1}$ if and only if $p_{n-(k-1)+i} < p_{n-(k-1)+i+1}$, and $p_i < p_n$ if and only if $p_{n-(k-1)+i} < p_n$ since $p_i$ and $p_{n-(k-1)+i}$ are consecutive values.  By Proposition $\ref{CalculateTheFwdImage}$:
$$ p'_{2i+1} = p'_{2(n-(k-1)+i)+1} +\varepsilon $$
So for each $i \in \{0, 1, \ldots, 2k-2 \}$: $p'_i = p'_{2n-2k+2+i}+\varepsilon$, and $\phi(p)$ can be decomposed as
$$\phi(p) = \pi_T[2a,2a+2n] = [\alpha'_1 \cdots \alpha'_{2k-1} \lambda'_1 \cdots \lambda'_{2l+1} \beta'_1 \cdots \beta'_{2k-1}],$$
where $\alpha'_i = \beta'_i + \varepsilon$, so $\phi(p)=p'$ is of type $2k-1$.

Next, suppose that $k \geq 2$ so $2k-1 \geq 3$, we show that $\phi_L(p) = L(p')$ and $\phi_R(p) = R(p')$ are of type $2k-2$ and $\phi_M(p)$ is of type $2k-3$.

Let $i \in \{0, 1, \ldots, 2k-3\}$, and consider $\phi_L(p) = L(p')$.  Since $p'_i$ and $p'_{2n-2k+2+i}$ are consecutive values, $p'_i < p'_{2n}$ if and only if $p'_{2n-2k+2+i} < p'_{2n}$.  So if $L(p')_i = p'_i$ then $L(p')_{2n-2k+2+i} = p'_{2n-2k+2+i}$, and if $L(p')_i = p'_i-1$ then $L(p')_{2n-2k+2+i} = p'_{2n-2k+2+i}-1$.  In either case, $L(p')_i = L(p')_{2n-2k+2+i} + \varepsilon$ and there is a decomposition 
$$\phi_L(p) = \pi_T[2a,2a+2n-1] = [\alpha'_1 \cdots \alpha'_{2k-2} \lambda'_1 \cdots \lambda'_{2l+2} \beta'_1 \cdots \beta'_{2k-2}],$$
and $\phi_L(p)$ is of type $2k-2$.

Now consider $\phi_R(p) = R(p')$.  Since $p'_{i+1}$ and $p'_{2n-2k+2+i+1}$ are consecutive values, $p'_{i+1} < p'_0$ if and only if $p'_{2n-2k+2+i+1} < p'_0$.  So if $R(p')_i = p'_{i+1}$ then $R(p')_{2n-2k+2+i} = p'_{2n-2k+2+i+1}$, and if $R(p')_i = p'_{i+1}-1$ then $R(p')_{2n-2k+2+i} = p'_{2n-2k+2+i+1}-1$.  In either case, $R(p')_i = R(p')_{2n-2k+2+i} + \varepsilon$ and there is a decomposition 
$$\phi_R(p) = \pi_T[2a+1,2a+2n] = [\alpha'_1 \cdots \alpha'_{2k-2} \lambda'_1 \cdots \lambda'_{2l+2} \beta'_1 \cdots \beta'_{2k-2}],$$
and $\phi_R(p)$ is of type $2k-2$.

Now consider $\phi_M(p)$, and let $i \in \{0, 1, \ldots, 2k-4\}$.  Since $R(p')_i$ and $R(p')_{2n-2k+1+i}$ are consecutive values; $R(p')_i < R(p')_{2n-1}$ if and only if $R(p')_{2n-2k+1+i} < R(p')_{2n-1}$.  So if $M(p')_i = L(R(p'))_i = R(p')_i$ then $M(p')_{2n-2k+1+i} = L(R(p'))_{2n-2k+1+i} = R(p')_{2n-2k+1+i}$, and if $M(p')_i = L(R(p'))_i = R(p')_i-1$ then $M(p')_{2n-2k+1+i} = L(R(p'))_{2n-2k+1+i} = R(p')_{2n-2k+1+i}-1$.  In either case, $M(p')_i = M(p')_{2n-2k+2+i} + \varepsilon$ and there is a decomposition 
$$\phi_M(p) = \pi_T[2a+1,2a+2n-1] = [\alpha'_1 \cdots \alpha'_{2k-3} \lambda'_1 \cdots \lambda'_{2l+3} \beta'_1 \cdots \beta'_{2k-3}],$$
and $\phi_R(p)$ is of type $2k-3$.

$\vspace{1.0ex}$

$\textbf{(b)}$  From (a), $\phi(q) = q'$ is of type $2k-1$.  Since $p$ and $q$ are a complementary pair of type $k$, $p_i = p_{n-k+1+i} + \varepsilon = q_i + \varepsilon = q_{n-k+1+i}$ for each $i \in \{0, 1, \ldots, k-1 \}$, and $p_{k+i} = q_{k+i}$ for each $i \in \{0, 1, \ldots, l-1 \}$.  Thus for $i \in \{0, 1, \ldots, k-1 \}$:
\begin{align*}
p'_{2i} &= p'_{2(n-k+1 + i)} + \varepsilon \\
p'_{2i} &= q'_{2(n-k+1 + i)} \\
q'_{2(n-k+1+i)} &= q'_{2i}+ \varepsilon
\end{align*}
For $i \in \{0, 1, \ldots, k-2 \}$:
\begin{align*}
p'_{2i+1} &= p'_{2(n-k+1+i)+1} + \varepsilon \\
p'_{2i+1} &= q'_{2(n-k+1+i)+1} \\
q'_{2(n-k+1+i)+1} &= q'_{2i+1}+ \varepsilon
\end{align*}

We know that $p_{k-1} = p_n + \varepsilon = q_{k-1} + \varepsilon = q_n$, so $p_{k-1} > p_n$ and $q_{k-1} < q_n$.  Thus if $p_{k-1} < p_k$
$$ p'_{2k-1} =  p_{k-1} + \abs{u}_1 + n = q_{k-1} + 1 + \abs{u}_1 + n = q_{k-1} + 1 + \abs{u}_1 + (n+1) = q'_{2k-1} $$
and if $p_{k-1} > p_k$
$$ p'_{2k-1} =  p_{k-1} + \abs{u}_1 - (n+1) = q_{k-1} + 1 + \abs{u}_1 - (n+1) = q_{k-1} + \abs{u}_1 - n = q'_{2k-1}. $$

By Proposition $\ref{CalculateTheFwdImage}$, since $p_{k+i} = q_{k+i}$ for each $i \in \{0, 1, \ldots, l-1 \}$, 
\begin{align*}
p'_{2(k+i)} &= q'_{2(k+i)} \\
p'_{2(k+i)+1} &= q'_{2(k+i)+1} 
\end{align*}

Thus there are decompositions of $\phi(p) = p'$ and $\phi(q) = q'$ so that
$$\phi(p) = \pi_T[2a,2a+2n] = [\alpha'_1 \cdots \alpha'_{2k-1} \lambda'_1 \cdots \lambda'_{2l+1} \beta'_1 \cdots \beta'_{2k-1}],$$
$$\phi(q) = \pi_T[2b,2b+2n] = [\beta'_1 \cdots \beta'_{2k-1} \lambda'_1 \cdots \lambda'_{2l+1} \alpha'_1 \cdots \alpha'_{2k-1}],$$
where $\alpha'_i = \beta'_i + \varepsilon$.  Therefore $\phi(p)=p'$ and $\phi(q)=q'$ are a complementary pair of type $2k-1$.

$\vspace{1.0ex}$

$\textbf{(c)}$  From (b), $\phi(p)=p'$ and $\phi(q)=q'$ are a complementary pair of type $2k-1$.  Suppose $k \geq 2$ and so $2k-3 \geq 1$, and let $i \in \{0, 1, \ldots, 2k-3\}$, then $p'_i = q'_i + \varepsilon = p'_{2n-2k+2 +i} + \varepsilon = q'_{2n-2k+2+i}$.  Thus $p'_i$ and $p'_{2n-2k+2 +i}$ are consecutive values, as are $q'_i$ and $q'_{2n-2k+2 +i}$, also $p'_{2n} < p'_i$ if and only if $p'_{2n} < p'_{2n-2k+2 +i}$, and
$$p'_{2n} < p'_i \text{ and } p'_{2n} < p'_{2n-2k+2 +i} \hspace{1.5ex} \Longleftrightarrow \hspace{1.5ex} q'_{2n} < q'_i \text{ and } q'_{2n} < q'_{2n-2k+2 +i}.$$
If $L(p')_i = p'_i-1$ or $L(p')_i = p'_i$, we have $L(q')_i = q'_i-1$ or $L(q')_i = q'_i$ (respectively), and $L(p')_i = L(q')_i+ \varepsilon = L(p')_{2n-2k+2 +i} + \varepsilon = L(q')_{2n-2k+2 +i}$.

Now let $i \in \{0, 1, \ldots, 2l \}$, so $p'_{2k-1+i} = q'_{2k-1+i}$.  Thus $p'_{2n} < p'_{2k-1+i}$ if and only if $q'_{2n} < q'_{2k-1+i}$, and so we have $L(p')_{2k-1+i} = L(q')_{2k-1+i}$.

Then $p'_{2k-2} = q'_{2k-2} + \varepsilon = p'_{2n} + \varepsilon = q'_{2n}$, so $p'_{2k-2} > p'_{2n}$ if and only if $q'_{2k-2} < q'_{2n}$.  If $p'_{2k-2} > p'_{2n}$ and $q'_{2k-2} < q'_{2n}$, then $p'_{2k-2} = q'_{2k-2} + 1 = p'_{2n} + 1 = q'_{2n}$ so
$$ L(p')_{2k-2} = p'_{2k-2} - 1 = q'_{2k-2} =  L(q')_{2k-2}.,$$
If $p'_{2k-2} < p'_{2n}$ and $q'_{2k-2} > q'_{2n}$, then $p'_{2k-2} = q'_{2k-2} - 1 = p'_{2n} - 1 = q'_{2n}$ so
$$ L(p')_{2k-2} = p'_{2k-2} = q'_{2k-2} - 1 =  L(q')_{2k-2}.$$
In either case, $ L(p')_{2k-2} =  L(q')_{2k-2}$.  Thus there are decompositions of $\phi_L(p) = L(p')$ and $\phi_L(q) = L(q')$ so that
$$\phi_L(p) = \pi_T[2a,2a+2n-1] = [\alpha'_1 \cdots \alpha'_{2k-2} \lambda'_1 \cdots \lambda'_{2l+2} \beta'_1 \cdots \beta'_{2k-2}],$$
$$\phi_L(q) = \pi_T[2b,2b+2n-1] = [\beta'_1 \cdots \beta'_{2k-2} \lambda'_1 \cdots \lambda'_{2l+2} \alpha'_1 \cdots \alpha'_{2k-2}],$$
where $\alpha'_i = \beta'_i + \varepsilon$.  Therefore $\phi_L(p)$ and $\phi_L(q)$ are a complementary pair of type $2k-2$.


Now suppose that $k=1$ and so $2k-1 = 1$.  Then $p'_0 = q'_0 + \varepsilon = p'_{2n} + \varepsilon= q'_{2n}$ and $p'_i = q'_i$ for $i = 1,2, \ldots, 2n-1$.  If $p'_0 > p'_{2n}$ and $q'_0 < q'_{2n}$, then $p'_0 = q'_0 + 1 = p'_{2n} + 1 = q'_{2n}$ so
$$ L(p')_0 = p'_0 - 1 = q'_0 =  L(q')_0.$$
If $p'_0 < p'_{2n}$ and $q'_0 > q'_{2n}$, then $p'_0 = q'_0 - 1 = p'_{2n} - 1 = q'_{2n}$ so
$$ L(p')_0 = p'_0 = q'_0 - 1 =  L(q')_0.$$
In either case, $ L(p')_0 =  L(q')_0$.  Then for each $i \in \{1, 2, \ldots, 2n-1 \}$, $p'_i = q'_i$, and $p'_{2n} < p'_i$ if and only if $q'_{2n} < q'_i$ so $L(p')_i = L(q')_i$.  Therefore, if $k=1$ then $\phi_L(p) = \phi_L(q)$.

$\vspace{1.0ex}$

$\textbf{(d)}$  From (b), $\phi(p)=p'$ and $\phi(q)=q'$ are a complementary pair of type $2k-1$.  Suppose $k \geq 2$ and so $2k-3 \geq 1$, and let $i \in \{0, 1, \ldots, 2k-3\}$, then $p'_{i+1} = q'_{i+1} + \varepsilon = p'_{2n-2k+2 +i+1} + \varepsilon = q'_{2n-2k+2+i+1}$.  Thus $p'_{i+1}$ and $p'_{2n-2k+2 +i+1}$ are consecutive values, as are $q'_{i+1}$ and $q'_{2n-2k+2 +i+1}$, also $p'_{2n} < p'_{i+1}$ if and only if $p'_{2n} < p'_{2n-2k+2 +i+1}$, and
$$p'_0 < p'_{i+1} \text{ and } p'_0 < p'_{2n-2k+2 +i+1} \hspace{1.5ex} \Longleftrightarrow \hspace{1.5ex} q'_0 < q'_{i+1} \text{ and } q'_0 < q'_{2n-2k+2 +i+1}.$$
If $R(p')_i = p'_{i+1}-1$ or $R(p')_i = p'_{i+1}$, we have $R(q')_i = q'_{i+1}-1$ or $R(q')_i = q'_{i+1}$ (respectively), and $R(p')_i = R(q')_i+ \varepsilon = R(p')_{2n-2k+2 +i} + \varepsilon = R(q')_{2n-2k+2 +i}$.

Now let $i \in \{0, 1, \ldots, 2l \}$, so $p'_{2k-1+i} = q'_{2k-1+i}$.  Thus $p'_0 < p'_{2k-1+i}$ if and only if $q'_0 < q'_{2k-1+i}$, and so we have $R(p')_{2k-1+i-1} = R(q')_{2k-1+i-1}$.

Then $p'_0 = q'_0 + \varepsilon = p'_{2n-2k+2} + \varepsilon = q'_{2n-2k+2}$, so $p'_{2n-2k+2} > p'_0$ if and only if $q'_{2n-2k+2} < q'_0$.  If $p'_{2n-2k+2} > p'_0$ and $q'_{2n-2k+2} < q'_0$, then $p'_{2n-2k+2} = q'_{2n-2k+2} + 1 = p'_0 + 1 = q'_0$ so
$$ R(p')_{2n-2k+1} = p'_{2n-2k+2} - 1 = q'_{2n-2k+2} =  R(q')_{2n-2k+1}.$$
If $p'_{2n-2k+2} < p'_0$ and $q'_{2n-2k+2} > q'_0$, then $p'_{2n-2k+2} = q'_{2n-2k+2} - 1 = p'_0 - 1 = q'_0$ so
$$ R(p')_{2n-2k+1} = p'_{2n-2k+2} = q'_{2n-2k+2} - 1 =  R(q')_{2n-2k+1}.$$
In either case, $ R(p')_{2n-2k+1} =  R(q')_{2n-2k+1}$.  Thus there are decompositions of $\phi_R(p) = R(p')$ and $\phi_R(q) = R(q')$ so that
$$\phi_L(p) = \pi_T[2a+1,2a+2n] = [\alpha'_1 \cdots \alpha'_{2k-2} \lambda'_1 \cdots \lambda'_{2l+2} \beta'_1 \cdots \beta'_{2k-2}],$$
$$\phi_L(q) = \pi_T[2b+1,2b+2n] = [\beta'_1 \cdots \beta'_{2k-2} \lambda'_1 \cdots \lambda'_{2l+2} \alpha'_1 \cdots \alpha'_{2k-2}],$$
where $\alpha'_i = \beta'_i + \varepsilon$.  Therefore $\phi_R(p)$ and $\phi_R(q)$ are a complementary pair of type $2k-2$.


Now suppose that $k=1$ and so $2k-1 = 1$.  Then $p'_0 = q'_0 + \varepsilon = p'_{2n} + \varepsilon= q'_{2n}$ and $p'_i = q'_i$ for $i = 1,2, \ldots, 2n-1$.  If $p'_0 > p'_{2n}$ and $q'_0 < q'_{2n}$, then $p'_0 = q'_0 + 1 = p'_{2n} + 1 = q'_{2n}$ so
$$ R(p')_{2n-1} = p'_{2n} = q'_{2n}-1 =  R(q')_{2n-1}.$$
If $p'_0 < p'_{2n}$ and $q'_0 > q'_{2n}$, then $p'_0 = q'_0 - 1 = p'_{2n} - 1 = q'_{2n}$ so
$$ R(p')_{2n-1} = p'_{2n} - 1 = q'_{2n} =  R(q')_{2n-1}.$$
In either case, $ R(p')_0 =  R(q')_0$.  Then for each $i \in \{1, 2, \ldots, 2n-1 \}$, $p'_i = q'_i$, and $p'_0 < p'_i$ if and only if $q'_0 < q'_i$ so $R(p')_{i-1} = R(q')_{i-1}$.  Therefore, if $k=1$ then $\phi_R(p) = \phi_R(q)$.

$\vspace{1.0ex}$

$\textbf{(e)}$  From (c), $\phi_R(p)=R(p')$ and $\phi_R(q)=R(q')$ are a complementary pair of type $2k-2$.  Suppose $k \geq 2$ and so $2k-4 \geq 0$, and let $i \in \{0,  \ldots, 2k-4\}$, then $R(p')_i = R(q')_i + \varepsilon = R(p')_{2n-2k+3 +i} + \varepsilon = R(q')_{2n-2k+3+i}$.  Thus $R(p')_i$ and $R(p')_{2n-2k+3 +i}$ are consecutive values, as are $R(q')_i$ and $R(q')_{2n-2k+3 +i}$, and $R(p')_{2n-1} < R(p')_i$ if and only if $R(p')_{2n-1} < R(p')_{2n-2k+3 +i}$, and
$$R(p')_{2n-1} < R(p')_i \text{ and } R(p')_{2n-1} < R(p')_{2n-2k+3 +i} \hspace{1.5ex} \Longleftrightarrow \hspace{1.5ex} R(q')_{2n-1} < R(q')_i \text{ and } R(q')_{2n-1} < R(q')_{2n-2k+3 +i}.$$
If $L(R(p'))_i = R(p')_i-1$ or $L(R(p'))_i = R(p')_i$, we have $L(R(q'))_i = R(q')_i-1$ or $L(R(q'))_i = R(q')_i$ (respectively), and $L(R(p'))_i = L(R(q'))_i+ \varepsilon = L(R(p'))_{2n-2k+2 +i} + \varepsilon = L(R(q'))_{2n-2k+2 +i}$.

Now let $i \in \{0, 1, \ldots, 2l+1 \}$, so $R(p')_{2k-2+i} = R(q')_{2k-2+i}$.  Thus $R(p')_{2n-1} < R(p')_{2k-2+i}$ if and only if $R(q')_{2n-1} < R(q')_{2k-2+i}$, and so we have $L(R(p'))_{2k-1+i} = L(R(q'))_{2k-1+i}$.

Then $R(p')_{2k-3} = R(q')_{2k-3} + \varepsilon = R(p')_{2n-1} + \varepsilon = R(q')_{2n-1}$, so $R(p')_{2k-3} > R(p')_{2n-1}$ if and only if $R(q')_{2k-3} < R(q')_{2n-1}$.  If $R(p')_{2k-3} > R(p')_{2n-1}$ and $R(q')_{2k-3} < R(q')_{2n-1}$, then $R(p')_{2k-3} = R(q')_{2k-3} + 1 = R(p')_{2n-1} + 1 = R(q')_{2n-1}$ so
$$ L(R(p'))_{2k-3} = R(p')_{2k-3} - 1 = R(q')_{2k-3} =  L(R(q'))_{2k-3}.$$
If $R(p')_{2k-3} < R(p')_{2n-1}$ and $R(q')_{2k-3} > R(q')_{2n-1}$, then $R(p')_{2k-3} = R(q')_{2k-3} - 1 = R(p')_{2n-1} - 1 = R(q')_{2n-1}$
$$ L(R(p'))_{2k-3} = R(p')_{2k-3} = R(q')_{2k-3} - 1 =  L(R(q'))_{2k-2-1}.$$
In either case, $ L(R(p'))_{2k-3} =  L(q')_{2k-3}$.  Thus there are decompositions of $\phi_M(p) = L(R(p'))$ and $\phi_M(q) = L(R(q'))$ so that
$$\phi_M(p) = \pi_T[2a-1,2a+2n-1] = [\alpha'_1 \cdots \alpha'_{2k-3} \lambda'_1 \cdots \lambda'_{2l+3} \beta'_1 \cdots \beta'_{2k-3}],$$
$$\phi_M(q) = \pi_T[2b-1,2b+2n-1] = [\beta'_1 \cdots \beta'_{2k-3} \lambda'_1 \cdots \lambda'_{2l+3} \alpha'_1 \cdots \alpha'_{2k-3}],$$
where $\alpha'_i = \beta'_i + \varepsilon$.  Therefore $\phi_M(p)$ and $\phi_M(q)$ are a complementary pair of type $2k-3$.


Now suppose that $k=1$ and so $2k-1 = 1$.  Then $\phi_R(p) = \phi_R(q)$, and thus $L(R(p')) = L(R(q'))$.  Therefore, if $k=1$ then $\phi_M(p) = \phi_M(q)$.
$\qed$
\end{proof}

\begin{theorem}
\label{SameFormIFFCompPair}
Let $p$ and $q$ be distinct subpermutations of $\pi_T$.  Then $p$ and $q$ have the same form if and only if $p$ and $q$ are a complementary pair of type $k$, for some $k \geq 1$.
\end{theorem}
\begin{proof}
First, suppose that $p$ and $q$ are a complementary pair of type $k$, for some $k \geq 1$.  So there are decompositions:
\begin{align*}
p &= \pi_T[a,a+n] = [\alpha_1 \cdots \alpha_k \lambda_1 \cdots \lambda_l \beta_1 \cdots \beta_k] \\ 
q &= \pi_T[b,b+n] = [\beta_1 \cdots \beta_k \lambda_1 \cdots \lambda_l \alpha_1 \cdots \alpha_k]  
\end{align*}
so that for $\varepsilon \in \{-1, 1 \}$, $\alpha_i = \beta_i + \varepsilon$ for each $i \in \{1, 2, \ldots, k \}$.  

For each $i \in \{0, 1, \ldots, k-2 \}$, $p_i$ and $p_{n-k+1+i}$ are consecutive values, as are $q_i$ and $q_{n-k+1+i}$, so 
$$ p_i < p_{i+1} \text{ and } p_{n-k+1+i} < p_{n-k+1+i+1} \hspace{1.5ex} \Longleftrightarrow \hspace{1.5ex} q_i < q_{i+1} \text{ and } q_{n-k+1+i} < q_{n-k+1+i+1}.$$
Since $p_{k-1} = q_{k-1} + \varepsilon$, $p_{k+l} + \varepsilon = q_{k+l}$, $p_k = q_k$, and $p_{k+l-1} = q_{k+l-1}$:
\begin{align*}
p_{k-1} < p_k \hspace{1.5ex} &\Longleftrightarrow \hspace{1.5ex} q_{k-1} < q_k \\ 
p_{k+l-1} < p_{k+l} \hspace{1.5ex} &\Longleftrightarrow \hspace{1.5ex} q_{k+l-1} < q_{k+l}. 
\end{align*}
For each $i \in \{0, 1, \ldots, l-2 \}$, $p_{k+i} = q_{k+i}$, so
$$ p_{k+i} < p_{k+i+1} \hspace{1.5ex} \Longleftrightarrow \hspace{1.5ex} q_{k+i} < q_{k+i+1}. $$
Therefore $p_i < p_{i+1}$ if and only if $q_i < q_{i+1}$ for each $i \in \{0, 1, \ldots, n-1 \}$, so $p$ and $q$ have the same form.

$\vspace{1.0ex}$

To show that distinct subpermutations with the same form are a complementary pair of type $k$, for some $k \geq 1$, an induction argument will be used.  The subpermutations of lengths 2 through 9 are listed in Appendix $\ref{SecTheSubperms}$, along with the form of the subpermutations.  It can be seen that distinct subpermutations with the same form are a complementary pair of type $k$, for some $k \geq 1$.  

Assume that $n \geq 9$ and that the theorem is true for all subpermutations of length at most $n$.  Let $p'$ and $q'$ be distinct subpermutations of length $n+1$ with the same form, so $p'_i < p_{i+1}$ if and only if $q'_i < q'_{i+1}$ for each $i = 0, 1, \ldots, n-1$.  

Then 
$$p', q' \in Perm_{ev}(n+1) \hspace{3.0ex} \text{ or } \hspace{3.0ex} p', q' \in Perm_{odd}(n+1).$$
If, without loss of generality, $p' \in Perm_{ev}(n+1)$ and $q' \in Perm_{odd}(n+1)$, then $p' = \pi_T[2a,2a+n]$ and $q' = \pi_T[2b+1,2b+n+1]$, so $T[2a,2a+n-1] = T[2b+1,2b+n]$.  Since $n \geq 9$, $T[2a,2a+n-1]$ will contain either 00 or 11, so there is some $c$ so that $T[2a+2c+1,2a+2c+2]$ is 00 or 11.  Then also, $T[2b+1 + 2c+1,2b+1 + 2c+2] = T[2b+2c+2,2b+2c+3]$ must be the same as $T[2a+2c+1,2a+2c+2]$, but $T[2b+2c+2,2b+2c+3]$ is either $\mu_T(0) = 01$ or $\mu_T(1) = 10$, so $T[2b+2c+2,2b+2c+3] \neq T[2a+2c+1,2a+2c+2]$.  Therefore, either $p',q' \in Perm_{ev}(n+1)$ or $p',q' \in Perm_{odd}(n+1)$

Thus one of the 4 following cases must hold:
\begin{enumerate}
\item $p',q' \in Perm_{ev}(n+1)$ and $n+1$ is odd
\item $p',q' \in Perm_{ev}(n+1)$ and $n+1$ is even
\item $p',q' \in Perm_{odd}(n+1)$and $n+1$ is even
\item $p',q' \in Perm_{odd}(n+1)$and $n+1$ is odd
\end{enumerate}

$\vspace{0.5ex}$

$\textbf{Case 1}$  Suppose $p',q' \in Perm_{ev}(n+1)$ and $n+1 = 2m+1$, so there are numbers $a$ and $b$ so that $p' = \pi_T[2a,2a+2m]$ and $q' = \pi_T[2b,2b+2m]$, and
$$ p = \pi_T[a,a+m] \hspace{8.0ex} q = \pi_T[b,b+m], $$ 
$$ p' = \phi(p) \hspace{8.0ex} q' = \phi(q). $$
If $T[a,a+m-1] \neq T[b,b+m-1]$ then $T[2a,2a+2m-1] \neq T[2b,2b+2m-1]$.  Hence 
$$T[a,a+m-1] = T[b,b+m-1]$$
and $p$ and $q$ have the same form.  If $p=q$ then $p'=q'$, by Lemma $\ref{pISqIFFppISqp}$, thus $p \neq q$.  By the induction hypothesis, $p$ and $q$ are a complementary pair of type $k$, for some $k \geq 1$.  Therefore, by Proposition $\ref{ImageOfTypeK}$, $\phi(p) = p'$ and $\phi(q) = q'$ are a complementary pair of type $2k-1$.

$\vspace{0.5ex}$

$\textbf{Case 2}$  Suppose $p',q' \in Perm_{ev}(n+1)$ and $n+1 = 2m$, so there are numbers $a$ and $b$ so that $p' = \pi_T[2a,2a+2m-1]$ and $q' = \pi_T[2b,2b+2m-1]$, and
$$ p = \pi_T[a,a+m] \hspace{8.0ex} q = \pi_T[b,b+m], $$
$$ p' = \phi_L(p) \hspace{8.0ex} q' = \phi_L(q). $$

Since $p'$ and $q'$ have the same form, $T[2a,2a+2m-2] = T[2b,2b+2m-2]$.  Thus $T_{2a+2m-2} = T_{2b+2m-2}$ implies $T_{a+m-1} = T_{b+m-1}$, so 
$$ T[2a+2m-2,2a+2m-1] = \mu_T(T_{a+m-1}) = \mu_T(T_{b+m-1}) = T[2b+2m-2,2b+2m-1]$$
and
$$T[2a,2a+2m-1] = T[2b,2b+2m-1].$$
If $T[a,a+m-1] \neq T[b,b+m-1]$ then $T[2a,2a+2m-1] \neq T[2b,2b+2m-1]$.  Hence 
$$T[a,a+m-1] = T[b,b+m-1]$$
and $p$ and $q$ have the same form.  If $p=q$ then $\phi(p)=\phi(q)$, by Lemma $\ref{pISqIFFppISqp}$, and $p'=L(\phi(p))=L(\phi(q))=q'$, thus $p \neq q$.  By the induction hypothesis, $p$ and $q$ are a complementary pair of type $k$, for some $k \geq 1$.  If $k=1$, then $\phi_L(p)$ and $\phi(q)_L$ are a complementary pair of type $2k-2 = 0$ and $p' = q'$, thus $k \geq 2$.  Therefore, by Proposition $\ref{ImageOfTypeK}$, $\phi_L(p) = p'$ and $\phi_L(q) = q'$ are a complementary pair of type $2k-2 \geq 2$.

$\vspace{0.5ex}$

$\textbf{Case 3}$  Suppose $p',q' \in Perm_{odd}(n+1)$ and $n+1 = 2m$, so there are numbers $a$ and $b$ so that $p' = \pi_T[2a+1,2a+2m]$ and $q' = \pi_T[2b+1,2b+2m]$, and
$$ p = \pi_T[a,a+m] \hspace{8.0ex} q = \pi_T[b,b+m], $$
$$ p' = \phi_R(p) \hspace{8.0ex} q' = \phi_R(q). $$

Since $p'$ and $q'$ have the same form, $T[2a+1,2a+2m-1] = T[2b+1,2b+2m-1]$.  Thus $T_{2a+1} = T_{2b+1}$ implies $T_{a} = T_{b}$, so 
$$ T[2a,2a+1] = \mu_T(T_{a}) = \mu_T(T_{b}) = T[2b,2b+1]$$
and
$$T[2a,2a+2m-1] = T[2b,2b+2m-1].$$
If $T[a,a+m-1] \neq T[b,b+m-1]$ then $T[2a,2a+2m-1] \neq T[2b,2b+2m-1]$.  Hence 
$$T[a,a+m-1] = T[b,b+m-1]$$
and $p$ and $q$ have the same form.  If $p=q$ then $\phi(p)=\phi(q)$, by Lemma $\ref{pISqIFFppISqp}$, and $p'=R(\phi(p))=R(\phi(q))=q'$, thus $p \neq q$.  By the induction hypothesis, $p$ and $q$ are a complementary pair of type $k$, for some $k \geq 1$.  If $k=1$, then $\phi_R(p)$ and $\phi_R(q)$ are a complementary pair of type $2k-2 = 0$ and $p' = q'$, thus $k \geq 2$.  Therefore, by Proposition $\ref{ImageOfTypeK}$, $\phi_R(p) = p'$ and $\phi(q)_R = q'$ are a complementary pair of type $2k-2 \geq 2$.

$\vspace{0.5ex}$

$\textbf{Case 4}$  Suppose $p',q' \in Perm_{odd}(n+1)$ and $n+1 = 2m+1$, so there are numbers $a$ and $b$ so that $p' = \pi_T[2a+1,2a+2m+1]$ and $q' = \pi_T[2b+1,2b+2m+1]$, and
$$ p = \pi_T[a,a+m+1] \hspace{8.0ex} q = \pi_T[b,b+m+1], $$
$$ p' = \phi_M(p) \hspace{8.0ex} q' = \phi_M(q). $$

Since $p'$ and $q'$ have the same form, $T[2a+1,2a+2m] = T[2b+1,2b+2m]$.  Thus $T_{2a+1} = T_{2b+1}$ implies $T_{a} = T_{b}$, so 
$$ T[2a,2a+1] = \mu_T(T_{a}) = \mu_T(T_{b}) = T[2b,2b+1]$$
and $T_{2a+2m} = T_{2b+2m}$ implies $T_{a+m} = T_{b+m}$, so 
$$ T[2a+2m,2a+2m+1] = \mu_T(T_{a+m}) = \mu_T(T_{b+m}) = T[2b+2m,2b+2m+1].$$
Therefore,
$$T[2a,2a+2m+1] = T[2b,2b+2m+1].$$
If $T[a,a+m] \neq T[b,b+m]$ then $T[2a,2a+2m+1] \neq T[2b,2b+2m+1]$.  Hence 
$$T[a,a+m] = T[b,b+m]$$
and $p$ and $q$ have the same form.  If $p=q$ then $\phi(p)=\phi(q)$, by Lemma $\ref{pISqIFFppISqp}$, and $p'=M(\phi(p))=M(\phi(q))=q'$, thus $p \neq q$.  By the induction hypothesis, $p$ and $q$ are a complementary pair of type $k$, for some $k \geq 1$.  If $k=1$, then $\phi_M(p)$ and $\phi_M(q)$ are a complementary pair of type $2k-3 = -1$ and $p' = q'$, thus $k \geq 2$.  Therefore, by Proposition $\ref{ImageOfTypeK}$, $\phi_M(p) = p'$ and $\phi_M(q) = q'$ are a complementary pair of type $2k-3 \geq 1$.

Therefore subpermutations $p$ and $q$ have the same form if and only if $p$ and $q$ are a complementary pair of type $k$, for some $k \geq 1$.
$\qed$
\end{proof}

There are a number of useful corollaries of Theorem $\ref{SameFormIFFCompPair}$.  These corollaries give the number of subpermutations that can have the same form and show when the maps $\phi_L$, $\phi_R$, and $\phi_M$ are not injective.

\begin{corollary}
\label{AtMostOneCompliment}
For a subpermutation $p$ of $\pi_T$, there can be at most one subpermutation $q$ of $\pi_T$ so that $p$ and $q$ are a complementary pair.
\end{corollary}
\begin{proof}
Assume that $p$ is a subpermutation of $\pi_T$ so that $p$ and $q$ are a complementary pair of type $s$, and $p$ and $r$ are a complementary pair of type $t$.  Moreover, $s \neq t$, and thus $q \neq r$.  Then there are decompositions:
\begin{align*}
p &= \pi_T[a,a+n] = [\alpha_1 \cdots \alpha_s \lambda_1 \cdots \lambda_x \beta_1 \cdots \beta_s] \\ 
q &= \pi_T[b,b+n] = [\beta_1 \cdots \beta_s \lambda_1 \cdots \lambda_x \alpha_1 \cdots \alpha_s]  
\end{align*}
so that for $\varepsilon_s \in \{-1, 1 \}$, $\alpha_i = \beta_i + \varepsilon_s$ for each $i = 1, 2, \ldots, s$, and
\begin{align*}
p &= \pi_T[a,a+n] = [\alpha'_1 \cdots \alpha'_t \lambda'_1 \cdots \lambda'_y \beta'_1 \cdots \beta'_t] \\ 
r &= \pi_T[b,b+n] = [\beta'_1 \cdots \beta'_t \lambda'_1 \cdots \lambda'_y \alpha'_1 \cdots \alpha'_t]
\end{align*}
so that for $\varepsilon_t \in \{-1, 1 \}$, $\alpha'_i = \beta'_i + \varepsilon_t$ for each $i = 1, 2, \ldots, t$.  

Since $p$ and $q$ are a complementary pair they have the same form, as do $p$ and $r$.  Thus $q$ and $r$ are distinct subpermutations with the same form, so by Theorem $\ref{SameFormIFFCompPair}$ $q$ and $r$ are a complementary pair of type $k$, for some $k$.

If $\beta_1 = \beta'_1$ then $p_{n-s+1} = p_{n-t+1}$, but since $s \neq t$ this cannot happen.  Thus $\beta_1 \neq \beta'_1$ and $\varepsilon_s \neq \varepsilon_t$, so $\varepsilon_s = -\varepsilon_t$.  Hence
\begin{align*}
\alpha_1 = \beta_1 + \varepsilon_s \hspace{3.0ex} \Rightarrow \hspace{3.0ex} \beta_1 &= \alpha_1 - \varepsilon_s  \\
\alpha'_1 = \beta'_1 + \varepsilon_t \hspace{3.0ex} \Rightarrow \hspace{3.0ex} \beta'_1 = \alpha'_1 - \varepsilon_t \hspace{3.0ex} \Rightarrow \hspace{3.0ex} \beta'_1 &= \alpha_1 + \varepsilon_s.
\end{align*}
Therefore $q_0 \neq r_0 \pm 1$, and $q$ and $r$ are not a complementary pair, contradicting the assumption.
$\qed$
\end{proof}

The next corollary follows directly from Theorem $\ref{SameFormIFFCompPair}$ and Corollary $\ref{AtMostOneCompliment}$

\begin{corollary}
\label{AtMostTwoSubpermsWithSameForm}
For a factor $u$ of $T$, there are at most two subpermutations of $\pi_T$ with form $u$.
\end{corollary}

The next corollary shows when the maps $\phi_L(p)$, $\phi_R(p)$, and $\phi_M(p)$ are not injective.

\begin{corollary}
\label{WhenTheMapsFailToBeABijection}
For subpermutations $p = \pi_T[a,a+n]$ and $q = \pi_T[b,b+n]$, where $p \neq q$:
\begin{itemize}
\item[(a)] $\phi_L(p) = \phi_L(q)$ if and only if $p$ and $q$ are a complementary pair of type 1.
\item[(b)] $\phi_R(p) = \phi_R(q)$ if and only if $p$ and $q$ are a complementary pair of type 1.
\item[(c)] $\phi_M(p) = \phi_M(q)$ if and only if $p$ and $q$ are a complementary pair of type 1.
\end{itemize}
\end{corollary}
\begin{proof}
It should be clear for all three cases that if $p$ and $q$ are a complementary pair of type 1 then 
$$\phi_L(p) = \phi_L(q) \hspace{6.0ex} \phi_R(p) = \phi_R(q) \hspace{6.0ex} \phi_M(p) = \phi_M(q)$$
by Proposition $\ref{ImageOfTypeK}$.  For the three cases, let $p = \pi_T[a,a+n]$ and $q = \pi_T[b,b+n]$ and $p \neq q$.


$\textbf{(a)}$  Suppose $\phi_L(p) = \phi_L(q)$, so $\pi_T[2a,2a+2n-1] = \pi_T[2b,2b+2n-1]$ and $T[2a,2a+2n-2] = T[2b,2b+2n-2]$.  Thus $T_{2a+2n-2} = T_{2b+2n-2}$ implies $T_{a+n-1} = T_{b+n-1}$, so 
$$ T[2a+2n-2,2a+2n-1] = \mu_T(T_{a+n-1}) = \mu_T(T_{b+n-1}) = T[2b+2n-2,2b+2n-1]$$
and
$$T[2a,2a+2n-1] = T[2b,2b+2n-1].$$
If $T[a,a+n-1] \neq T[b,b+n-1]$ then $T[2a,2a+2n-1] \neq T[2b,2b+2n-1]$.  Hence 
$$T[a,a+n-1] = T[b,b+n-1]$$
and $p$ and $q$ have the same form.  By Theorem $\ref{SameFormIFFCompPair}$, $p$ and $q$ are a complementary pair of type $k \geq 1$.  If $k > 1$, then $\phi_L(p)$ and $\phi_L(q)$ are a complementary pair of type $2k-2 > 1$, so $\phi_L(p) \neq \phi_L(q)$.  Therefore $p$ and $q$ are a complementary pair of type 1. 

$\vspace{0.5ex}$

$\textbf{(b)}$  Suppose $\phi_R(p) = \phi_R(q)$, so $\pi_T[2a+1,2a+2n] = \pi_T[2b+1,2b+2n]$ and $T[2a+1,2a+2n-1] = T[2b+1,2b+2n-1]$.  Thus $T_{2a+1} = T_{2b+1}$ implies $T_a = T_b$, so 
$$ T[2a,2a+1] = \mu_T(T_{a}) = \mu_T(T_{b}) = T[2b,2b+1]$$
and
$$T[2a,2a+2n-1] = T[2b,2b+2n-1].$$
If $T[a,a+n-1] \neq T[b,b+n-1]$ then $T[2a,2a+2n-1] \neq T[2b,2b+2n-1]$.  Hence 
$$T[a,a+n-1] = T[b,b+n-1]$$
and $p$ and $q$ have the same form.  By Theorem $\ref{SameFormIFFCompPair}$, $p$ and $q$ are a complementary pair of type $k \geq 1$.  If $k > 1$, then $\phi_R(p)$ and $\phi_R(q)$ are a complementary pair of type $2k-2 > 1$, so $\phi_R(p) \neq \phi_R(q)$.  Therefore $p$ and $q$ are a complementary pair of type 1. 

$\vspace{0.5ex}$

$\textbf{(c)}$  Suppose $\phi_M(p) = \phi_M(q)$, so $\pi_T[2a+1,2a+2n-1] = \pi_T[2b+1,2b+2n-1]$ and $T[2a+1,2a+2n-2] = T[2b+1,2b+2n-2]$.  Thus $T_{2a+1} = T_{2b+1}$ implies $T_{a} = T_{b}$, so 
$$ T[2a,2a+1] = \mu_T(T_{a}) = \mu_T(T_{b}) = T[2b,2b+1]$$
and $T_{2a+2n} = T_{2b+2n}$ implies $T_{a+n} = T_{b+n}$, so 
$$ T[2a+2n,2a+2n+1] = \mu_T(T_{a+n}) = \mu_T(T_{b+n}) = T[2b+2n,2b+2n+1].$$
Therefore,
$$T[2a,2a+2n+1] = T[2b,2b+2n+1].$$
If $T[a,a+n] \neq T[b,b+n]$ then $T[2a,2a+2n+1] \neq T[2b,2b+2n+1]$.  Hence 
$$T[a,a+n] = T[b,b+n]$$
and $p$ and $q$ have the same form.  By Theorem $\ref{SameFormIFFCompPair}$, $p$ and $q$ are a complementary pair of type $k \geq 1$.  If $k > 1$, then $\phi_M(p)$ and $\phi_M(q)$ are a complementary pair of type $2k-3 \geq 1$, so $\phi_M(p) \neq \phi_M(q)$.  Therefore $p$ and $q$ are a complementary pair of type 1. 
$\qed$
\end{proof}

So when there are complementary pairs of type 1 none of the maps $\phi_L$, $\phi_R$, and $\phi_M$ are injective, and thus they are not bijective.  In cases where there are no complementary pairs of type 1 the maps $\phi_L$, $\phi_R$, and $\phi_M$ are  injective and the inequalities in Lemma $\ref{UpperBoundForTau}$ become equalities.  So we need to know when complementary pairs of type 1 will occur, and how many complementary pairs there are.

\section{Type 1 Pairs}
\label{SecTypeOnePairs}

This section investigates when complementary pairs of type 1 arise and the number of pairs that occur.  To show when the maps $\phi_L$, $\phi_R$, and $\phi_M$ are bijections we need to consider when complementary pairs of type 1 occur.  The following lemma shows when there are complementary pairs of type $k$, for each $k \geq 0$.  An induction argument will be used with Proposition $\ref{ImageOfTypeK}$ and Theorem $\ref{SameFormIFFCompPair}$ to show that all complementary pairs of a given length are of same type.

\begin{proposition}
\label{LengthOfAlphaForAGBForN}
Let $n > 4$ be a natural number and let $p$ and $q$ be subpermutations of $\pi_T$ of length $n+1$ with the same form.  There exist $r$ and $c$ so that $n =2^r+c$, where $0 \leq c < 2^r$.  
\begin{itemize}
\item[(a)] If $0 \leq c < 2^{r-1}+1$, then either $p = q$ or $p$ and $q$ are a complementary pair of type $c+1$.
\item[(b)] If $2^{r-1}+1 \leq c < 2^r$, then $p = q$.
\end{itemize}
\end{proposition}
\begin{proof}
This will be proved using an induction argument on $r$.  By looking at the subpermutations in Appendix $\ref{SecTheSubperms}$ it can be readily verified that the lemma is true for $r = 2$ and $c = 0, 1, 2, 3$, so for $n = 4, 5, 6, 7$.  Suppose that $r>2$ and that the statement of the lemma is true when $n < 2^r$.  It will be shown that it is true for all $n = 2^r+c$ where $0 \leq c < 2^r$.  

%
%
$\vspace{0.5ex}$

$\textbf{(a)}$  Let $n=2^r+c$ with $0 \leq c < 2^{r-1}+1$.  If $p' = q'$ the proposition is satisfied, so assume that $p' \neq q'$.  As it was stated in the proof of Theorem $\ref{SameFormIFFCompPair}$, if $p' \in Perm_{ev}(n+1)$ and $q' \in Perm_{odd}(n+1)$, then $p'$ and $q'$ cannot have the same form.  We must also consider when $n+1$ is both even and odd.  So there will be four subcases to consider, when $p',q' \in Perm_{ev}(n+1)$ or when $p',q' \in Perm_{odd}(n+1)$ and when $n+1$ is even or odd.

$\textbf{Case a.1:}$  Suppose $p',q' \in Perm_{ev}(n+1)$ and $n+1$ is odd, so $c$ is even.  There is a $d$ so that $c=2d$, with $0 \leq d < 2^{r-2}+1$, and there are numbers $a$ and $b$ so that $p' = \pi_T[2a,2a+2^r+2d]$ and $q' = \pi_T[2b,2b+2^r+2d]$, and
$$ p = \pi_T[a,a+2^{r-1}+d] \hspace{8.0ex} q = \pi_T[b,b+2^{r-1}+d], $$ 
$$ p' = \phi(p) \hspace{8.0ex} q' = \phi(q). $$
If $T[a,a+2^{r-1}+d-1] \neq T[b,b+2^{r-1}+d-1]$ then $T[2a,2a+2^r+2d-1] \neq T[2b,2b+2^r+2d-1]$.  Hence 
$$T[a,a+2^{r-1}+d-1] = T[b,b+2^{r-1}+d-1]$$
and $p$ and $q$ have the same form.  If $p=q$ then $p'=q'$, by Corollary $\ref{CorTo_pisqiff}$, thus $p \neq q$.  By the induction hypothesis, $p$ and $q$ are a complementary pair of type $d+1$.  Therefore, by Proposition $\ref{ImageOfTypeK}$, $\phi(p) = p'$ and $\phi(q) = q'$ are a complementary pair of type $2(d+1)-1 = 2d+1 = c+1$.

$\vspace{0.5ex}$

$\textbf{Case a.2:}$  Suppose $p',q' \in Perm_{odd}(n+1)$ and $n+1$ is odd, so $c$ is even.  There is a $d$ so that $c=2d$, with $0 \leq d < 2^{r-2}+1$, and there are numbers $a$ and $b$ so that $p' = \pi_T[2a+1,2a+2^r+2d+1]$ and $q' = \pi_T[2b+1,2b+2^r+2d+1]$, and
$$ p = \pi_T[a,a+2^{r-1}+d+1] \hspace{8.0ex} q = \pi_T[b,b+2^{r-1}+d+1], $$
$$ p' = \phi_M(p) \hspace{8.0ex} q' = \phi_M(q). $$

Since $p'$ and $q'$ have the same form, $T[2a+1,2a+2^r+2d] = T[2b+1,2b+2^r+2d]$.  Thus $T_{2a+1} = T_{2b+1}$ implies $T_{a} = T_{b}$, so 
$$ T[2a,2a+1] = \mu_T(T_{a}) = \mu_T(T_{b}) = T[2b,2b+1]$$
and $T_{2a+2^r+2d} = T_{2b+2^r+2d}$ implies $T_{a+2^{r-1}+d} = T_{b+2^{r-1}+d}$, so 
$$ T[2a+2^r+2d,2a+2^r+2d+1] = \mu_T(T_{a+2^{r-1}+d}) = \mu_T(T_{b+2^{r-1}+d}) = T[2b+2^r+2d,2b+2^r+2d+1].$$
Therefore,
$$T[2a,2a+2^r+2d+1] = T[2b,2b+2^r+2d+1].$$
If $T[a,a+2^{r-1}+d] \neq T[b,b+2^{r-1}+d]$ then $T[2a,2a+2^r+2d+1] \neq T[2b,2b+2^r+2d+1]$.  Hence 
$$T[a,a+2^{r-1}+d] = T[b,b+2^{r-1}+d]$$
and $p$ and $q$ have the same form.  If $p=q$ then $\phi(p)=\phi(q)$, by Corollary $\ref{CorTo_pisqiff}$, and $p'=M(\phi(p))=M(\phi(q))=q'$, thus $p \neq q$.  By the induction hypothesis, $p$ and $q$ are a complementary pair of type $d+2$.  Therefore, by Proposition $\ref{ImageOfTypeK}$, $\phi_M(p) = p'$ and $\phi_M(q) = q'$ are a complementary pair of type $2(d+2)-3 = 2d+1 = c+1$.

$\vspace{0.5ex}$

$\textbf{Case a.3:}$  Suppose $p',q' \in Perm_{ev}(n+1)$ and $n+1$ is even, so $c$ is odd.  There is a $d$ so that $c=2d+1$, with $0 \leq d < 2^{r-2}+1$, and there are numbers $a$ and $b$ so that $p' = \pi_T[2a,2a+2^r+2d+1]$ and $q' = \pi_T[2b,2b+2^r+2d+1]$, and
$$ p = \pi_T[a,a+2^{r-1}+d+1] \hspace{8.0ex} q = \pi_T[b,b+2^{r-1}+d+1], $$
$$ p' = \phi_L(p) \hspace{8.0ex} q' = \phi_L(q). $$

Since $p'$ and $q'$ have the same form, $T[2a,2a+2^r+2d] = T[2b,2b+2^r+2d]$.  Thus $T_{2a+2^r+2d} = T_{2b+2^r+2d}$ implies $T_{a+2^{r-1}+d} = T_{b+2^{r-1}+d}$, so 
$$ T[2a+2^r+2d,2a+2^r+2d+1] = \mu_T(T_{a+2^{r-1}+d}) = \mu_T(T_{b+2^{r-1}+d}) = T[2b+22^r+2d,2b+22^r+2d+1]$$
and
$$T[2a,2a+2^r+2d+1] = T[2b,2b+2^r+2d+1].$$
If $T[a,a+m-1] \neq T[b,b+m-1]$ then $T[2a,2a+2m-1] \neq T[2b,2b+2m-1]$.  Hence 
$$T[a,a+2^{r-1}+d] = T[b,b+2^{r-1}+d]$$
and $p$ and $q$ have the same form.  If $p=q$ then $\phi(p)=\phi(q)$, by Corollary $\ref{CorTo_pisqiff}$, and $p'=L(\phi(p))=L(\phi(q))=q'$, thus $p \neq q$.  By the induction hypothesis, $p$ and $q$ are a complementary pair of type $d+2$.  Therefore, by Proposition $\ref{ImageOfTypeK}$, $\phi_L(p) = p'$ and $\phi_L(q) = q'$ are a complementary pair of type $2(d+2)-2 = 2d+2 = c+1$.

$\vspace{0.5ex}$

$\textbf{Case a.4:}$  Suppose $p',q' \in Perm_{odd}(n+1)$ and $n+1$ is even, so $c$ is odd.  There is a $d$ so that $c=2d+1$, with $0 \leq d < 2^{r-2}+1$, and there are numbers $a$ and $b$ so that $p' = \pi_T[2a+1,2a+2^r+2d+2]$ and $q' = \pi_T[2b+1,2b+2^r+2d+2]$, and
$$ p = \pi_T[a,a+2^{r-1}+d+1] \hspace{8.0ex} q = \pi_T[b,b+2^{r-1}+d+1], $$
$$ p' = \phi_R(p) \hspace{8.0ex} q' = \phi_R(q). $$

Since $p'$ and $q'$ have the same form, $T[2a+1,2a+2^r+2d+1] = T[2b+1,2b+2^r+2d+1]$.  Thus $T_{2a+1} = T_{2b+1}$ implies $T_{a} = T_{b}$, so 
$$ T[2a,2a+1] = \mu_T(T_{a}) = \mu_T(T_{b}) = T[2b,2b+1]$$
and
$$T[2a,2a+2^r+2d+1] = T[2b,2b+2^r+2d+1].$$
If $T[a,a+2^{r-1}+d] \neq T[b,b+2^{r-1}+d]$ then $T[2a,2a+2^r+2d+1] \neq T[2b,2b+2^r+2d+1]$.  Hence 
$$T[a,a+2^{r-1}+d] = T[b,b+2^{r-1}+d]$$
and $p$ and $q$ have the same form.  If $p=q$ then $\phi(p)=\phi(q)$, by Corollary $\ref{CorTo_pisqiff}$, and $p'=R(\phi(p))=R(\phi(q))=q'$, thus $p \neq q$.  By the induction hypothesis, $p$ and $q$ are a complementary pair of type $d+2$.  Therefore, by Proposition $\ref{ImageOfTypeK}$, $\phi_R(p) = p'$ and $\phi(q)_R = q'$ are a complementary pair of type $2(d+2)-2 = 2d+2 = c+1$.

%
%
$\vspace{0.5ex}$

$\textbf{(b)}$  Let $n=2^r+c$ with $2^{r-1}+1 \leq c < 2^r$.  There will again be the four subcases from part $(a)$ when $2^{r-1}+1 \leq c < 2^r-2$, when $p',q' \in Perm_{ev}(n+1)$ or when $p',q' \in Perm_{odd}(n+1)$ and when $n+1$ is even or odd.  There will also be 2 additional special cases to consider, which are when $c = 2^r-2$ and $c=2^r-1$.

$\textbf{Case b.1:}$  Suppose $p',q' \in Perm_{ev}(n+1)$ and $n+1$ is odd, so $c$ is even.  There is a $d$ so that $c=2d$, with $2^{r-2}+1 \leq d < 2^{r-1}$, and there are numbers $a$ and $b$ so that $p' = \pi_T[2a,2a+2^r+2d]$ and $q' = \pi_T[2b,2b+2^r+2d]$, and
$$ p = \pi_T[a,a+2^{r-1}+d] \hspace{8.0ex} q = \pi_T[b,b+2^{r-1}+d], $$ 
$$ p' = \phi(p) \hspace{8.0ex} q' = \phi(q). $$
As in case $\textbf{a.1}$, $T[a,a+2^{r-1}+d-1] = T[b,b+2^{r-1}+d-1]$, so $p$ and $q$ have the same form.  By the induction hypothesis $p = q$, so by Corollary $\ref{CorTo_pisqiff}$, $p' = \phi(p) = \phi(q) = q'$.

$\vspace{0.5ex}$

$\textbf{Case b.2:}$  Suppose $p',q' \in Perm_{odd}(n+1)$ and $n+1$ is odd, so $c$ is even.  There is a $d$ so that $c=2d$, with $2^{r-2}+1 \leq d < 2^{r-1}$, and there are numbers $a$ and $b$ so that $p' = \pi_T[2a+1,2a+2^r+2d+1]$ and $q' = \pi_T[2b+1,2b+2^r+2d+1]$, and
$$ p = \pi_T[a,a+2^{r-1}+d+1] \hspace{8.0ex} q = \pi_T[b,b+2^{r-1}+d+1], $$
$$ p' = \phi_M(p) \hspace{8.0ex} q' = \phi_M(q). $$
As in case $\textbf{a.2}$, $T[a,a+2^{r-1}+d] = T[b,b+2^{r-1}+d]$, so $p$ and $q$ have the same form.  By the induction hypothesis $p = q$, so by Corollary $\ref{CorTo_pisqiff}$, $\phi(p) = \phi(q)$ and therefore $p' = \phi_M(p) = \phi_M(q) = q'$.

$\vspace{0.5ex}$

$\textbf{Case b.3:}$  Suppose $p',q' \in Perm_{ev}(n+1)$ and $n+1$ is even, so $c$ is odd.  There is a $d$ so that $c=2d+1$, with $2^{r-2}+1 \leq d < 2^{r-1}$, and there are numbers $a$ and $b$ so that $p' = \pi_T[2a,2a+2^r+2d+1]$ and $q' = \pi_T[2b,2b+2^r+2d+1]$, and
$$ p = \pi_T[a,a+2^{r-1}+d+1] \hspace{8.0ex} q = \pi_T[b,b+2^{r-1}+d+1], $$
$$ p' = \phi_L(p) \hspace{8.0ex} q' = \phi_L(q). $$
As in case $\textbf{a.3}$, $T[a,a+2^{r-1}+d] = T[b,b+2^{r-1}+d]$, so $p$ and $q$ have the same form.  By the induction hypothesis $p = q$, so by Corollary $\ref{CorTo_pisqiff}$, $\phi(p) = \phi(q)$ and therefore $p' = \phi_L(p) = \phi_L(q) = q'$.

$\vspace{0.5ex}$

$\textbf{Case b.4:}$  Suppose $p',q' \in Perm_{odd}(n+1)$ and $n+1$ is even, so $c$ is odd.  There is a $d$ so that $c=2d+1$, with $2^{r-2}+1 \leq d < 2^{r-1}$, and there are numbers $a$ and $b$ so that $p' = \pi_T[2a+1,2a+2^r+2d+2]$ and $q' = \pi_T[2b+1,2b+2^r+2d+2]$, and
$$ p = \pi_T[a,a+2^{r-1}+d+1] \hspace{8.0ex} q = \pi_T[b,b+2^{r-1}+d+1], $$
$$ p' = \phi_R(p) \hspace{8.0ex} q' = \phi_R(q). $$
As in case $\textbf{a.4}$, $T[a,a+2^{r-1}+d] = T[b,b+2^{r-1}+d]$, so $p$ and $q$ have the same form.  By the induction hypothesis $p = q$, so by Corollary $\ref{CorTo_pisqiff}$, $\phi(p) = \phi(q)$ and therefore $p' = \phi_R(p) = \phi_R(q) = q'$.

$\vspace{0.5ex}$

$\textbf{Case b.5:}$  Suppose $c = 2^r-2$.  Thus $n = 2^r + 2^r-2 = 2^{r+1} - 2$, and the subpermutations $p'$ and $q'$ will have odd length.  There will be two subcases, these being when $p',q' \in Perm_{ev}(n+1)$ and when $p',q' \in Perm_{odd}(n+1)$.

$\textbf{Case b.5.i:}$  Suppose $p',q' \in Perm_{ev}(n+1)$.  There are numbers $a$ and $b$ so that $p' = \pi_T[2a,2a+2^{r+1}-2]$ and $q' = \pi_T[2b,2b+2^{r+1}-2]$, and
$$ p = \pi_T[a,a+2^r-1] \hspace{8.0ex} q = \pi_T[b,b+2^r-1], $$ 
$$ p' = \phi(p) \hspace{8.0ex} q' = \phi(q). $$
As in cases $\textbf{a.1}$ and $\textbf{b.1}$, $T[a,a+2^r-2] = T[b,b+2^r-2]$, so $p$ and $q$ have the same form.  By the induction hypothesis $p = q$, so by Corollary $\ref{CorTo_pisqiff}$, $p' = \phi(p) = \phi(q) = q'$.

$\textbf{Case b.5.ii:}$  Suppose $p',q' \in Perm_{odd}(n+1)$.  There are numbers $a$ and $b$ so that $p' = \pi_T[2a+1,2a+2^{r+1}-1]$ and $q' = \pi_T[2b+1,2b+2^{r+1}-1]$, and
$$ p = \pi_T[a,a+2^r] \hspace{8.0ex} q = \pi_T[b,b+2^r], $$ 
$$ p' = \phi_M(p) \hspace{8.0ex} q' = \phi_M(q). $$
As in cases $\textbf{a.2}$ and $\textbf{b.2}$, $T[a,a+2^r-1] = T[b,b+2^r-1]$, so $p$ and $q$ have the same form.  If $p=q$ then $\phi(p)=\phi(q)$, by Corollary $\ref{CorTo_pisqiff}$, and $p'=M(\phi(p))=M(\phi(q))=q'$.  If $p \neq q$ then by case $\textbf{a.1}$, $p$ and $q$ are a complementary pair of type 1.  Therefore, by Proposition $\ref{ImageOfTypeK}$, $p' = \phi_M(p) = \phi_M(q) = q'$.

$\vspace{0.5ex}$

$\textbf{Case b.6:}$  Suppose $c = 2^r-1$.  Thus $n = 2^r + 2^r-1 = 2^{r+1}-1$, and the subpermutations $p'$ and $q'$ will have even length.  There will be two subcases, these being when $p',q' \in Perm_{ev}(n+1)$ and when $p',q' \in Perm_{odd}(n+1)$.

$\textbf{Case b.6.i:}$  Suppose $p',q' \in Perm_{ev}(n+1)$.  There are numbers $a$ and $b$ so that $p' = \pi_T[2a,2a+2^{r+1}-1]$ and $q' = \pi_T[2b,2b+2^{r+1}-1]$, and
$$ p = \pi_T[a,a+2^r] \hspace{8.0ex} q = \pi_T[b,b+2^r], $$ 
$$ p' = \phi_L(p) \hspace{8.0ex} q' = \phi_L(q). $$
As in cases $\textbf{a.3}$ and $\textbf{b.3}$, $T[a,a+2^r-1] = T[b,b+2^r-1]$, so $p$ and $q$ have the same form.  If $p=q$ then $\phi(p)=\phi(q)$, by Corollary $\ref{CorTo_pisqiff}$, and $p'=L(\phi(p))=L(\phi(q))=q'$.  If $p \neq q$ then by case $\textbf{a.1}$, $p$ and $q$ are a complementary pair of type 1.  Therefore, by Proposition $\ref{ImageOfTypeK}$, $p' = \phi_L(p) = \phi_L(q) = q'$.

$\textbf{Case b.6.ii:}$  Suppose $p',q' \in Perm_{odd}(n+1)$.  There are numbers $a$ and $b$ so that $p' = \pi_T[2a+1,2a+2^{r+1}]$ and $q' = \pi_T[2b+1,2b+2^{r+1}]$, and
$$ p = \pi_T[a,a+2^r] \hspace{8.0ex} q = \pi_T[b,b+2^r], $$ 
$$ p' = \phi_R(p) \hspace{8.0ex} q' = \phi_R(q). $$
As in cases $\textbf{a.4}$ and $\textbf{b.4}$, $T[a,a+2^r-1] = T[b,b+2^r-1]$, so $p$ and $q$ have the same form.  If $p=q$ then $\phi(p)=\phi(q)$, by Corollary $\ref{CorTo_pisqiff}$, and $p'=R(\phi(p))=R(\phi(q))=q'$.  If $p \neq q$ then by case $\textbf{a.1}$, $p$ and $q$ are a complementary pair of type 1.  Therefore, by Proposition $\ref{ImageOfTypeK}$, $p' = \phi_R(p) = \phi_R(q) = q'$.

Therefore the lemma is true when $n=2^r+c$ with $0 \leq c < 2^r$, and therefore for all $n$.
$\qed$
\end{proof}

Thus, only subpermutations of length $2^r+1$ can be a complementary pair of type 1, and we have the following corollary.

\begin{corollary}
\label{BijectionForTheMaps}
If $n \neq 2^r$, for $r \geq 1$, then for any subpermutations $p = \pi_T[a,a+n]$ and $q = \pi_T[b,b+n]$
\begin{itemize}
\item[(a)] $\phi_L(p) = \phi_L(q)$ if and only if $p = q$.
\item[(b)] $\phi_R(p) = \phi_R(q)$ if and only if $p = q$.
\item[(c)] $\phi_M(p) = \phi_M(q)$ if and only if $p = q$.
\end{itemize}
\end{corollary}
\begin{proof}
It should be clear in each case that if $p = q$ then 
$$\phi_L(p) = \phi_L(q) \hspace{5.0ex} \phi_R(p) = \phi_R(q) \hspace{5.0ex} \phi_M(p) = \phi_M(q).$$


Suppose $\phi_L(p) = \phi_L(q)$.  If $p \neq q$, by Corollary $\ref{WhenTheMapsFailToBeABijection}$ $p$ and $q$ are a complementary pair of type 1.  By Proposition $\ref{LengthOfAlphaForAGBForN}$, $p$ and $q$ are cannot be complementary pair of type 1, therefore $p=q$.

A similar argument will show if $\phi_R(p) = \phi_R(q)$ then $p=q$, and if $\phi_M(p) = \phi_M(q)$ then $p=q$.
%
%
%
$\qed$
\end{proof}

We now consider the number of factors $u$ of $T$ of length $2^r$ that have two subpermutations which form a complementary pair of type 1.

\begin{lemma}
\label{NumberOfFactWithLenAlphaOne}
Let $n  = 2^r$ or $2^r+1$, with $r \geq 2$.  Then there are exactly $2^r$ factors $u$ of $T$ of length $n$ so that there exist subpermutations $p = \pi_T[a,a+n]$ and $q = \pi_T[b,b+n]$ with form $u$ and $p \neq q$.
\end{lemma}
\begin{proof}
It can be readily verified by looking at the subpermutations in Appendix $\ref{SecTheSubperms}$ that the lemma is true for $r=2$.  So there are 4 factors $u$ of $T$ of length 4 with two distinct subpermutations of length 5 with form $u$, and there are 4 factors $v$ of $T$ of length 5 with two distinct subpermutations of length 6 with form $v$.  

Suppose $r \geq 2$ and that the lemma is true for $r$.  We now show the lemma is true for $r+1$.  Let $\Gamma$ be the set of factors of length $2^r$, $\abs{\Gamma} = 2^r$, so that for $u \in \Gamma$ there are subpermutations $p$ and $q$ with form $u$ so that $p \neq q$, hence, by Proposition $\ref{LengthOfAlphaForAGBForN}$, $p$ and $q$ are a complementary pair of type 1.  Let $\Gamma'$ be the set of factors of length $2^{r+1}$ so that if $u \in \Gamma'$ then there exist subpermutations $p$ and $q$ with form $u$ so that $p \neq q$.  Let $\Delta$ be the set of factors of length $2^r+1$, $\abs{\Delta} = 2^r$, so that for $v \in \Delta$ there are subpermutations $p$ and $q$ with form $v$ so that $p \neq q$, hence, by Proposition $\ref{LengthOfAlphaForAGBForN}$, $p$ and $q$ are a complementary pair of type 2.  Let $\Delta'$ be the set of factors of length $2^{r+1}+1$ so that if $v \in \Delta'$ then there exist subpermutations $p$ and $q$ with form $v$ so that $p \neq q$.

The sizes of $\Gamma'$ and $\Delta'$ will be considered in two cases.

%
%

$\vspace{0.5ex}$

$\textbf{Case $\Gamma'$:}$  Any factor in $\Gamma'$ will either start in an even position or an odd position, call these sets of factors $\Gamma'_{ev}$ and $\Gamma'_{odd}$ and hence
$$ \Gamma' = \Gamma'_{ev} \cup \Gamma'_{odd} .$$
Since the factors are of length $2^{r+1} \geq 8$, for any factors $s \in \Gamma'_{ev}$ and $t \in \Gamma'_{odd}$, $s \neq t$, thus
$$ \Gamma'_{ev} \cap \Gamma'_{odd} = \emptyset.$$
There will be two subcases to establish the size of $\Gamma'$, first by showing the size of $\Gamma'_{ev}$ and then the size of $\Gamma'_{odd}$.


$\textbf{Subcase $\Gamma'_{ev}$:}$  For $u \in \Gamma$ there are subpermutations $p$ and $q$ of $\pi_T$ of length $2^r+1$, so that $p \neq q$.  By Proposition $\ref{LengthOfAlphaForAGBForN}$, $p$ and $q$ are a complementary pair of type 1.  By Proposition $\ref{ImageOfTypeK}$ $\phi(p)$ and $\phi(q)$ are a complementary pair of type 1, so $\phi(p) \neq \phi(q)$ and they both have form $\mu_T(u)$.  Therefore for each $u \in \Gamma$, $\mu_T(u) \in \Gamma'_{ev}$.  Hence 
$$ \abs{\Gamma'_{ev}} \geq \abs{\Gamma}.$$

Suppose that $u' \in \Gamma'_{ev}$, so there are subpermutations $p' = \pi_T[2a, 2a+2^{r+1}]$ and $q' = \pi_T[2b, 2b+2^{r+1}]$ with form $u' = T[2a, 2a+2^{r+1}-1] = T[2b, 2b+2^{r+1}-1]$, so that $p' \neq q'$.  Hence there exist subpermutations $p$ and $q$ so that $\phi(p) = p'$ and  $\phi(q) = q'$.  As in case $\textbf{a.1}$ of Proposition $\ref{LengthOfAlphaForAGBForN}$, $p$ and $q$ are a complementary pair of type 1 with form $u$ where $\mu_T(u) = u'$.  Thus for each $u' \in \Gamma'_{ev}$, there is some $u \in \Gamma$ so that $\mu_T(u) = u'$.  Hence
$$ \abs{\Gamma'_{ev}} \leq \abs{\Gamma}.$$

Therefore $\abs{\Gamma'_{ev}} = \abs{\Gamma}$.


$\textbf{Subcase $\Gamma'_{odd}$:}$  For $u \in \Delta$, $u = T[a,a+2^r]$ , there are subpermutations $p$ and $q$ of $\pi_T$ of length $2^r+2$, so that $p \neq q$.  By Proposition $\ref{LengthOfAlphaForAGBForN}$, $p$ and $q$ are a complementary pair of type 2.  By Proposition $\ref{ImageOfTypeK}$, $\phi(p)$ and $\phi(q)$ are a complementary pair of type 3 with form $\mu_T(u) = T[2a,2a+2^{r+1}+1]$ and $\phi_M(p)$ and $\phi_M(q)$ are a complementary pair of type 1, so $\phi_M(p) \neq \phi_M(q)$ and they both have form $T[2a+1,2a+2^{r+1}]$.  Therefore for each $T[a,a+2^r] \in \Delta$, $T[2a+1,2a+2^{r+1}] \in \Gamma'_{odd}$.  Hence
$$ \abs{\Gamma'_{odd}} \geq \abs{\Delta}.$$

Suppose that $u' \in \Gamma'_{odd}$, so there are subpermutations $p' = \pi_T[2a+1, 2a+2^{r+1}+1]$ and $q' = \pi_T[2b+1, 2b+2^{r+1}+1]$ with form $u' = T[2a+1, 2a+2^{r+1}] = T[2b+1, 2b+2^{r+1}]$, so that $p' \neq q'$.  Hence there exist subpermutations $p$ and $q$ so that $\phi_M(p) = p'$ and  $\phi_M(q) = q'$.  As in case $\textbf{a.2}$ of Proposition $\ref{LengthOfAlphaForAGBForN}$, $p$ and $q$ are a complementary pair of type 2 with form $T[a,a+2^r]$.  Thus for each $u' \in \Gamma'_{odd}$, there is some $T[a,a+2^r] \in \Delta$ so that $u' = T[2a+1, 2a+2^{r+1}]$.  Hence
$$ \abs{\Gamma'_{odd}} \leq \abs{\Delta}.$$

Therefore $\abs{\Gamma'_{odd}} = \abs{\Delta}$.


Therefore
$$\abs{\Gamma'} =  \abs{\Gamma'_{ev}} +\abs{\Gamma'_{odd}} = \abs{\Gamma} + \abs{\Delta} = 2^r + 2^r = 2^{r+1}. $$

%
%

$\vspace{0.5ex}$

$\textbf{Case $\Delta'$:}$  Any factor in $\Delta'$ will either start in an even position or an odd position, call these sets of factors $\Delta'_{ev}$ and $\Delta'_{odd}$ and hence
$$ \Delta' = \Delta'_{ev} \cup \Delta'_{odd} .$$
Since the factors are of length $2^{r+1}+1 \geq 8$, for any factors $s \in \Delta'_{ev}$ and $t \in \Delta'_{odd}$, $s \neq t$, thus
$$ \Delta'_{ev} \cap \Delta'_{odd} = \emptyset.$$
There will be two subcases to establish the size of $\Delta'$, first by showing the size of $\Delta'_{ev}$ and then the size of $\Delta'_{odd}$.


$\textbf{Subcase $\Delta'_{ev}$:}$  For $u \in \Delta$, $u = T[a,a+2^r]$ , there are subpermutations $p$ and $q$ of $\pi_T$ of length $2^r+2$, so that $p \neq q$.  By Proposition $\ref{LengthOfAlphaForAGBForN}$, $p$ and $q$ are a complementary pair of type 2.  By Proposition $\ref{ImageOfTypeK}$, $\phi(p)$ and $\phi(q)$ are a complementary pair of type 3 with form $\mu_T(u) = T[2a,2a+2^{r+1}+1]$ and $\phi_L(p)$ and $\phi_L(q)$ are a complementary pair of type 2, so $\phi_L(p) \neq \phi_L(q)$ and they both have form $T[2a,2a+2^{r+1}]$.  Therefore for each $T[a,a+2^r] \in \Delta$, $T[2a,2a+2^{r+1}] \in \Delta'_{ev}$.  Hence
$$ \abs{\Delta'_{ev}} \geq \abs{\Delta}.$$

Suppose that $u' \in \Delta'_{ev}$, so there are subpermutations $p' = \pi_T[2a+1, 2a+2^{r+1}+2]$ and $q' = \pi_T[2b+1, 2b+2^{r+1}+2]$ with form $u' = T[2a+1, 2a+2^{r+1}+1] = T[2b+1, 2b+2^{r+1}+1]$, so that $p' \neq q'$.  Hence there exist subpermutations $p$ and $q$ so that $\phi_L(p) = p'$ and  $\phi_L(q) = q'$.  As in case $\textbf{a.3}$ of Proposition $\ref{LengthOfAlphaForAGBForN}$, $p$ and $q$ are a complementary pair of type 2 with form $u = T[a,a+2^r]$.  Thus for each $u' \in \Gamma'_{ev}$, there is some $T[a,a+2^r] \in \Delta$ so that $u' = T[2a+1, 2a+2^{r+1}+1]$.  Hence
$$ \abs{\Delta'_{ev}} \leq \abs{\Delta}.$$

Therefore $\abs{\Delta'_{ev}} = \abs{\Delta}$.


$\textbf{Subcase $\Delta'_{odd}$:}$  A symmetric argument to the argument used in Subcase $\Delta'_{ev}$ will show $\abs{\Delta'_{odd}} = \abs{\Delta}$.

Therefore
$$\abs{\Delta'} =  \abs{\Delta'_{ev}} +\abs{\Delta'_{odd}} = \abs{\Delta} + \abs{\Delta} = 2^r + 2^r = 2^{r+1}. $$
$\qed$
\end{proof}

Now we know when there are complementary pairs of type 1, and how many pairs of type 1 there are in each case.  

\section{Permutation Complexity of $T$}
\label{FormulaForPermComp}

We are now ready to give a recursive definition for the permutation complexity of $T$.  To show this we consider when the maps $\phi$, $\phi_L$, $\phi_R$, and $\phi_M$ are bijective.  After the recursive definition is given, it will be shown that the recursive definition yields a formula for the permutation complexity.

\begin{proposition}
\label{RecursivePermComp}
Let $n \in \N$.  When $2n+1 = 2^r-1$, for some $r \geq 3$:
$$ \tau_T(2n+1) = \tau_T(n+1) + \tau_T(n+2) - 2^{r-1}. $$
When $2n = 2^r$, for some $r \geq 3$:
$$ \tau_T(2n) = 2(\tau_T(n+1) - 2^{r-1}). $$
For all other $n \geq 3$:
\begin{align*}
\tau_T(2n+1) &= \tau_T(n+1) + \tau_T(n+2) \\
\tau_T(2n) &= 2(\tau_T(n+1)).
\end{align*}
\end{proposition}
\begin{proof}
For any $n$, 
$$ \tau_T(n) = \abs{Perm(n)} = \abs{Perm_{ev}(n)} + \abs{Perm_{odd}(n)}. $$
This proof will be done in three cases.  The first is when $2n+1 = 2^r-1$ for some $r \geq 3$, the second is when $2n = 2^r$ for some $r \geq 3$, and the third for all other $n$.


$\textbf{Case $2n+1 = 2^r-1$: }$  It can be readily verified by looking at the subpermutations in Appendix $\ref{SecTheSubperms}$ that the proposition is true for $r=3$.  Suppose $r \geq 3$ and the lemma is true for $r$.  We show that the lemma is true for $r+1$.  So $2n+1 = 2^{r+1}-1$, and 
$$ Perm(2n+1) = Perm_{ev}(2n+1) + Perm_{odd}(2n+1). $$
Since the map 
$$\phi: Perm(n+1) \rightarrow Perm_{ev}(2n+1)$$
is a bijection, the size of $Perm(n+1)$ is the same as the size of $Perm_{ev}(2n+1)$.  Therefore
$$ \abs{Perm_{ev}(2n+1)} = \abs{Perm(n+1)} = \tau_T(n+1). $$

Then the map
$$ \phi_M: Perm(n+2) \rightarrow Perm_{odd}(2n+1) $$
is a surjective map, so
$$ \abs{Perm_{odd}(2n+1)} \leq \abs{Perm(n+2)}, $$
but it is not injective because $n+2 = 2^r + 1$.  So there are $2^r$ factors $u$ of length $2^r$ with a complementary pair of type 1 by Proposition $\ref{LengthOfAlphaForAGBForN}$ and Lemma $\ref{NumberOfFactWithLenAlphaOne}$.  Thus there are exactly $2^{r}$ complementary pairs of type 1 in $Perm(n+2)$.  So $2^{r+1}$ subpermutations in $Perm(n+2)$ will be mapped to $2^r$ subpermutations in $Perm_{odd}(2n+1)$ under $\phi_M$.  The other $Perm(n+2) - 2^{r+1}$ subpermutations in $Perm(n+2)$ are pairwise distinct and not complementary pairs, and thus will be pairwise distinct under $\phi_M$.  Hence
$$ \abs{Perm_{odd}(2n+1)} = \left( \abs{Perm(n+2)} - 2^{r+1} \right) + 2^r = \tau_T(n+2) - 2^r. $$

Therefore
$$ \tau_T(n) = \tau_T(n+1) + \tau_T(n+2) - 2^r.  $$

$\vspace{0.5ex}$

$\textbf{Case $2n+1 = 2^r$: }$  It can be readily verified by looking at the subpermutations in Appendix $\ref{SecTheSubperms}$ that the proposition is true for $r=3$.  Suppose $r \geq 3$ and the lemma is true for $r$, and we show that the lemma is true for $r+1$.  So $2n+1 = 2^{r+1}$, and 
$$ Perm(2n) = Perm_{ev}(2n) + Perm_{odd}(2n). $$
The map
$$ \phi_L: Perm(n+1) \rightarrow Perm_{ev}(2n) $$
is a surjective map, so
$$ \abs{Perm_{ev}(2n)} \leq \abs{Perm(n+1)}, $$
but it is not injective because $n+1 = 2^r + 1$.  So there are $2^r$ factors $u$ of length $2^r$ with a complementary pair of type 1 by Proposition $\ref{LengthOfAlphaForAGBForN}$ and Lemma $\ref{NumberOfFactWithLenAlphaOne}$.  Thus there are exactly $2^{r}$ complementary pairs of type 1 in $Perm(n+1)$.  So $2^{r+1}$ subpermutations in $Perm(n+1)$ will be mapped to $2^r$ subpermutations in $Perm_{ev}(2n)$ under $\phi_M$.  The other $Perm(n+1) - 2^{r+1}$ subpermutations in $Perm(n+1)$ are pairwise distinct and not complementary pairs, and thus will be pairwise distinct under $\phi_L$.  Hence
$$ \abs{Perm_{ev}(2n)} = \left( \abs{Perm(n+1)} - 2^{r+1} \right) + 2^r = \abs{Perm(n+1)} - 2^r. $$

The map
$$ \phi_R: Perm(n+1) \rightarrow Perm_{odd}(2n) $$
is a surjective map, so
$$ \abs{Perm_{odd}(2n)} \leq \abs{Perm(n+1)}, $$
but it is not injective because $n+1 = 2^r + 1$.  By a similar argument to above we can see that 
$$ \abs{Perm_{odd}(2n)} = \abs{Perm(n+1)} - 2^r. $$

Therefore
$$ \tau_T(n) = (\abs{Perm(n+1)} - 2^r) + (\abs{Perm(n+1)} - 2^r) = 2(\tau_T(n+1) - 2^r). $$

$\vspace{0.5ex}$

$\textbf{Case $n \geq 3$: }$  It can be readily verified by looking at the subpermutations in Appendix $\ref{SecTheSubperms}$ that the proposition is true for $n=3$.  Suppose $n \geq 3$ and the lemma is true for $n$, and we show that the lemma is true for $n+1$.  Since $2(n+1)+1, 2(n+1) \notin \{2^r-1, 2^r | r \geq 2\}$ for any $r$, we have $n+2, n+3 \notin \{ 2^r+1 | r \geq 2 \}$.  So for $2(n+1)$ and $2(n+1) + 1$ we know that the maps
\begin{align*}
&\phi: Perm(n+2) \rightarrow Perm_{ev}(2(n+1)+1) \\
&\phi_L: Perm(n+2) \rightarrow Perm_{ev}(2(n+1)) \\
&\phi_R: Perm(n+2) \rightarrow Perm_{odd}(2(n+1)) \\
&\phi_M: Perm(n+3) \rightarrow Perm_{odd}(2(n+1)+1) 
\end{align*}
are all bijections.  Therefore:
\begin{align*}
&\abs{Perm_{ev}(2(n+1)+1)} = \abs{Perm(n+2)} = \tau_T(n+2) \\
&\abs{Perm_{ev}(2(n+1))} = \abs{Perm(n+2)} = \tau_T(n+2) \\
&\abs{Perm_{odd}(2(n+1))} = \abs{Perm(n+2)} = \tau_T(n+2) \\
&\abs{Perm_{odd}(2(n+1)+1)} = \abs{Perm(n+3)} = \tau_T(n+3).
\end{align*}
So:
\begin{align*}
&\tau_T(2(n+1)) = \abs{Perm_{ev}(2(n+1))} + \abs{Perm_{odd}(2(n+1))} = 2( \tau_T(n+2)) \\
&\tau_T(2(n+1)+1) = \abs{Perm_{ev}(2(n+1)+1)} + \abs{Perm_{odd}(2(n+1)+1)} = \tau_T(n+2) + \tau_T(n+3).
\end{align*}
$\qed$
\end{proof}

\begin{theorem}
\label{PermCompIsTheFormula}
For any $n \geq 6$, where $n = 2^a + b$ with $0 < b \leq 2^a$,
$$ \tau_T(n) = 2(2^{a+1}+b-2).$$
\end{theorem}
\begin{proof}
The proof will be done by induction on $n$.  The above formula can be readily verified by looking at the subpermutations listed in Appendix $\ref{SecTheSubperms}$ for $n \leq 9$.  Suppose the theorem is true for all values less than or equal to $2n$.

$\textbf{Case $2n+1 = 2^a-1$: }$ Suppose $2n+1 = 2^a-1$.  If $2n+1 = 2^a-1 = 2^{a-1}+2^{a-1} - 1$, then $n = 2^{a-1}-1$, so $n+1 = 2^{a-1} = 2^{a-2}+2^{a-2}$ and $n+2=2^{a-1}+1$.  Thus:
\begin{align*}
\tau_T(n+1) &= 2(2^{a-2+1}+2^{a-2}-2) = 2(2^{a-1}+2^{a-2}-2) = 2(3(2^{a-2})-2)\\
\tau_T(n+2) &= 2(2^{a-1+1}+1-2) = 2(2^a-1)
\end{align*}
From Proposition $\ref{RecursivePermComp}$:
\begin{align*}
\tau_T(2n+1) &= 2(3(2^{a-2})-2) + 2(2^a-1) - 2^{a-1} = 2(3(2^{a-2})-2 + 2^a-1 - 2^{a-2}) \\
 &= 2(2(2^{a-2}) + 2^a-3) = 2(2^a + (2^{a-1}-1) - 2)
\end{align*}

$\vspace{0.5ex}$

$\textbf{Case $2n+2 = 2(n+1) = 2^a$: }$  Suppose $2n+2 = 2(n+1) = 2^a = 2^{a-1}+2^{a-1}$:
\begin{align*}
\tau_T(2(n+1)) &= 2(2(2^a-1) - 2^{a-1}) = 2(2^{a+1} - 2^{a-1} - 2) = 2(3(2^{a-1}) - 2)\\
 &= 2(2(2^{a-1}) + 2^{a-1} - 2) = 2(2^a + 2^{a-1} - 2)
\end{align*}

$\vspace{0.5ex}$

$\textbf{Case Else: }$ Suppose $2n+1 = 2^a + b$, $2n+2 = 2(n+1) = 2^a + b+1$, and $0 < b < 2^a - 1$.  Since $2n+1 = 2^a + b$ is odd, $b$ is odd.  So $n = 2^{a-1}+\frac{b-1}{2}$, $n+1 = 2^{a-1}+\frac{b+1}{2}$, and $n+2 = 2^{a-1}+\frac{b+3}{2}$.  Thus:
\begin{align*}
\tau_T(n+1) &= 2(2^a + \frac{b+1}{2} -2)\\
\tau_T(n+2) &= 2(2^a + \frac{b+3}{2} -2).
\end{align*}
From Proposition $\ref{RecursivePermComp}$:
\begin{align*}
\tau_T(2n+1) &=  2(2^a + \frac{b+1}{2} -2) + 2(2^a + \frac{b+3}{2} -2)  = 2(2^a + 2^a + \frac{b+1}{2}+ \frac{b+3}{2} -2 -2)\\
 &=  2(2^{a+1} + \frac{2b+4}{2} -4) = 2(2^{a+1} + b -2) \\
 \\
\tau_T(2(n+1)) &= 2(2(2^a + \frac{b+3}{2} -2)) = 2(2^{a+1} + b+3 -4)\\
 &= 2(2^{a+1} + (b+1) -2).
\end{align*}

Therefore, for all $n \geq 6$, where $n = 2^a + b$ with $0 < b \leq 2^a$, $ \tau_T(n) = 2(2^{a+1}+b-2)$
$\qed$
\end{proof}

\section{Conclusion}
\label{SecConclusion}

There seem to be some natural ways to continue this research.  For the binary doubling map $\delta$, defined as $\delta(0) =00$ and $\delta(1)=11$, it has been shown that $T$ and $\delta(T)$ have the same factor complexity ($\cite{AberBrle02}$).  One natural question is, do $T$ and $\delta(T)$ have the same permutation complexity?  The answer is no.  As can be seen in Appendix $\ref{SecTheSubperms}$, $\tau_T(5) = 14$ but $\tau_{\delta(T)}(5) = 16$.  With $T$, there are at most two distinct subpermutations that have the same, but with $\delta(T)$ there are cases where three subpermutations have the same form.  One open question is, what is the permutation complexity of $\delta(T)$?

This paper also investigates the action of the $\mu_T$ on the subpermutations of $\pi_T$.  Since $\mu_T$ is an order preserving map, we know that if there are distinct subpermutations $\pi_T[a,a+n]$ and $\pi_T[b,b+n]$ then $\pi_T[2a,2a+2n] \neq \pi_T[2b,2b+2n]$.  This seems to be true in general for binary words that are fixed points of morphisms by using a similar argument from Lemma $\ref{pISqIFFppISqp}$, but the converse is not true in general.   Another open question is to investigate properties of infinite permutations associated with aperiodic binary words that are fixed points of a morphism.  For such words, is there a way to define a mapping on the subpermutations of $\pi_\w$ similar to the map $\phi$ defined on the subpermutations of $\pi_T$?

These are only a couple of the open questions in the area of permutation complexity.

%
\paragraph{Acknowledgements:} 
Steve Widmer thanks Luca Zamboni and Amy Glen for comments and suggestions that helped him to improve and clarify this paper.

%
\appendix
\section{Subpermutations of $\pi_T$} 
\label{SecTheSubperms}

The subpermutations and their form for factors of length 1 through 8 are shown below.

$$ 0 : [1 2] \hspace{5.0ex} 1 : [2 1] $$
\begin{align*}
01 : [1 3 2] \hspace{2.0ex} [2 3 1] \hspace{4.0ex} 00 : [1 2 3]& \\
10 : [3 1 2] \hspace{2.0ex} [2 1 3] \hspace{4.0ex} 11 : [3 2 1]& 
\end{align*}

\begin{align*}
010 : [2 4 1 3] \hspace{2.0ex} [1 3 2 4] \hspace{5.0ex} 001 : [1 2 4 3] \hspace{5.0ex} 100 : [3 1 2 4]& \\
101 : [4 2 3 1] \hspace{2.0ex} [3 1 4 2] \hspace{5.0ex} 011 : [2 4 3 1] \hspace{5.0ex} 110 : [4 3 1 2]& 
\end{align*}

\begin{align*}
0011 : [2 3 5 4 1] \hspace{2.0ex} [1 3 5 4 2]& \hspace{5.0ex} 0010 : [1 2 4 3 5] \hspace{5.0ex}  1010 : [5 2 4 1 3]\\
0110 : [2 5 4 1 3] \hspace{2.0ex} [3 5 4 1 2]& \hspace{5.0ex} 0100 : [2 4 1 3 5] \hspace{5.0ex}  1011 : [4 2 5 3 1]\\
1001 : [4 1 2 5 3] \hspace{2.0ex} [3 1 2 5 4]& \hspace{5.0ex} 0101 : [1 4 2 5 3] \hspace{5.0ex}  1101 : [5 4 2 3 1]\\
1100 : [5 3 1 2 4] \hspace{2.0ex} [4 3 1 2 5]& 
\end{align*}

\begin{align*}
00110 : [2 4 6 5 1 3] \hspace{2.0ex} [1 3 6 5 2 4] \hspace{5.0ex} 00101 : [1 2 5 3 6 4] \hspace{5.0ex} 10010 : [4 1 2 5 3 6]& \\
01100 : [3 6 4 1 2 5] \hspace{2.0ex} [2 5 4 1 3 6] \hspace{5.0ex} 01001 : [2 5 1 3 6 4] \hspace{5.0ex} 10100 : [5 2 4 1 3 6]& \\
10011 : [5 2 3 6 4 1] \hspace{2.0ex} [4 1 3 6 5 2] \hspace{5.0ex} 01011 : [2 5 3 6 4 1] \hspace{5.0ex} 10110 : [5 2 6 4 1 3]& \\
11001 : [6 4 1 2 5 3] \hspace{2.0ex} [5 3 1 2 6 4] \hspace{5.0ex} 01101 : [3 6 5 2 4 1] \hspace{5.0ex} 11010 : [6 5 2 4 1 3]& 
\end{align*}

\begin{align*}
011001 : [3 7 5 1 2 6 4] \hspace{2.0ex} [2 6 4 1 3 7 5] \hspace{5.0ex} 001011 : [1 3 6 4 7 5 2] \hspace{5.0ex} 100101 : [4 1 2 6 3 7 5]& \\
100110 : [6 2 4 7 5 1 3] \hspace{2.0ex} [5 1 3 7 6 2 4] \hspace{5.0ex} 001100 : [2 4 7 5 1 3 6] \hspace{5.0ex} 101001 : [6 2 5 1 3 7 4]& \\
001101 : [2 4 7 6 3 5 1] \hspace{5.0ex} 101100 : [5 2 6 4 1 3 7]& \\
010010 : [2 5 1 3 6 4 7] \hspace{5.0ex} 101101 : [6 3 7 5 2 4 1]& \\
010011 : [3 6 2 4 7 5 1] \hspace{5.0ex} 110010 : [6 4 1 2 5 3 7]& \\
010110 : [2 6 3 7 5 1 4] \hspace{5.0ex} 110011 : [6 4 1 3 7 5 2]& \\
011010 : [4 7 6 2 5 1 3] \hspace{5.0ex} 110100 : [7 5 2 4 1 3 6]& 
\end{align*}

\begin{align*}
0010110 : [1 3 7 4 8 6 2 5] \hspace{5.0ex} 0101100 : [2 6 3 7 5 1 4 8] \hspace{5.0ex} 1001011 : [5 1 3 7 4 8 6 2] \hspace{5.0ex} 1011001 : [6 2 7 4 1 3 8 5]&\\
0011001 : [2 4 8 6 1 3 7 5] \hspace{5.0ex} 0101101 : [3 7 4 8 6 2 5 1] \hspace{5.0ex} 1001100 : [6 2 4 8 5 1 3 7] \hspace{5.0ex} 1011010 : [7 4 8 6 2 5 1 3]&\\
0011010 : [2 5 8 7 3 6 1 4] \hspace{5.0ex} 0110010 : [3 7 5 1 2 6 4 8] \hspace{5.0ex} 1001101 : [6 2 4 8 7 3 5 1] \hspace{5.0ex} 1100101 : [7 4 1 2 6 3 8 5]&\\
0100101 : [2 5 1 3 7 4 8 6] \hspace{5.0ex} 0110011 : [3 7 5 1 4 8 6 2] \hspace{5.0ex} 1010010 : [6 2 5 1 3 7 4 8] \hspace{5.0ex} 1100110 : [7 5 1 3 8 6 2 4]&\\
0100110 : [3 7 2 5 8 6 1 4] \hspace{5.0ex} 0110100 : [4 8 6 2 5 1 3 7] \hspace{5.0ex} 1010011 : [7 3 6 2 4 8 5 1] \hspace{5.0ex} 1101001 : [8 6 2 5 1 3 7 4]&
\end{align*}

\begin{align*}
00101101 : [2 4 8 5 9 7 3 6 1]  \hspace{2.0ex}  [1 4 8 5 9 7 3 6 2]& \hspace{5.0ex} 00101100 : [1 3 7 4 8 6 2 5 9] \hspace{5.0ex} 10011001 : [7 2 4 9 6 1 3 8 5] \\
01001011 : [3 6 1 4 8 5 9 7 2]  \hspace{2.0ex}  [2 6 1 4 8 5 9 7 3]& \hspace{5.0ex} 00110010 : [2 4 8 6 1 3 7 5 9] \hspace{5.0ex} 10011010 : [7 2 5 9 8 3 6 1 4] \\
01011010 : [4 8 5 9 7 2 6 1 3]  \hspace{2.0ex}  [3 8 5 9 7 2 6 1 4]& \hspace{5.0ex} 00110100 : [2 5 9 7 3 6 1 4 8] \hspace{5.0ex} 10100110 : [8 3 7 2 5 9 6 1 4] \\
01101001 : [4 9 7 2 6 1 3 8 5]  \hspace{2.0ex}  [5 9 7 2 6 1 3 8 4]& \hspace{5.0ex} 01001100 : [3 7 2 5 9 6 1 4 8] \hspace{5.0ex} 10110011 : [7 3 8 5 1 4 9 6 2] \\
10010110 : [6 1 3 8 4 9 7 2 5]  \hspace{2.0ex}  [5 1 3 8 4 9 7 2 6]& \hspace{5.0ex} 01011001 : [2 7 3 8 5 1 4 9 6] \hspace{5.0ex} 11001011 : [8 5 1 3 7 4 9 6 2] \\
10100101 : [7 2 5 1 3 8 4 9 6]  \hspace{2.0ex}  [6 2 5 1 3 8 4 9 7]& \hspace{5.0ex} 01100101 : [3 8 5 1 2 7 4 9 6] \hspace{5.0ex} 11001101 : [8 6 2 4 9 7 3 5 1] \\
10110100 : [8 4 9 6 2 5 1 3 7]  \hspace{2.0ex}  [7 4 9 6 2 5 1 3 8]& \hspace{5.0ex} 01100110 : [3 8 6 1 4 9 7 2 5] \hspace{5.0ex} 11010011 : [9 7 3 6 2 4 8 5 1] \\
11010010 : [9 6 2 5 1 3 7 4 8]  \hspace{2.0ex}  [8 6 2 5 1 3 7 4 9]&
\end{align*}

\bibliographystyle{plain}
\bibliography{mybib}

\end{document}